\documentclass[11pt]{amsart}
\usepackage[utf8]{inputenc}
\usepackage[a4paper,margin=3cm]{geometry}
\usepackage[pdftex,pdfborder={0 0 0},colorlinks,pagebackref]{hyperref} 
\usepackage{latexsym,lmodern,graphicx,enumitem,amssymb,dsfont}
\usepackage{xcolor}
\usepackage{mathrsfs}

\newcommand{\eps}{\varepsilon}

\newcommand{\N}{\mathbb{N}}
\newcommand{\R}{\mathbb{R}}

\newcommand{\E}{\mathbb{E}}

\newcommand{\X}{\mathcal{X}}

\theoremstyle{definition}

\newcommand{\sto}{\mathrm{sto}}

\newcommand{\leqc}{\preceq_{c}} 
\newcommand{\leqs}{\preceq_\sto} 
\newcommand{\leqcs}{\preceq_{c,\sto}} 
\newcommand{\dg}{\mathsf{d}}
\newcommand{\Id}{\mathrm{Id}}

\newcommand{\B}{\mathcal{B}}

\theoremstyle{plain}
\newtheorem{thm}{Theorem}[section]
\newtheorem{lem}{Lemma}[section]
\newtheorem{prop}{Proposition}[section]
\newtheorem{cor}{Corollary}[section]
\newtheorem{defi}{Definition}[section]
\newtheorem{rem}{Remark}[section]

\definecolor{boubelcolor}{rgb}{.65,0.05,0}

\DeclareUnicodeCharacter{2260}{\colorlet{memoireboubel}{.}\color{boubelcolor}} 
\DeclareUnicodeCharacter{00B1}{\color{memoireboubel}{}} 
\DeclareUnicodeCharacter{00A3}{\colorlet{memoirejuillet}{.}\color{blue}} 
\DeclareUnicodeCharacter{00A7}{\color{memoirejuillet}{}} 

\begin{document}
\title[On a mixture of Brenier and Strassen Theorems]{On a mixture of Brenier and Strassen Theorems}
\author{Nathael Gozlan}
\author{Nicolas Juillet}
\date{\today}

\date{\today}

\address{
NG : Universit\'e Paris Descartes, MAP5, UMR 8145, 45 rue des Saints Pères, 75270 Paris Cedex 06}
\email{natael.gozlan@parisdescartes.fr}
\address{
NJ: Universit\'e de Strasbourg, IRMA, UMR 7501, 7 rue René-Descartes 67084 Strasbourg Cedex}
\email{nicolas.juillet@math.unistra.fr}


\maketitle
\begin{abstract} We give a characterization of optimal transport plans for a variant of the usual quadratic transport cost introduced in \cite{GRST17}. Optimal plans are composition of a deterministic transport given by the gradient of a continuously differentiable convex function followed by a martingale coupling. We also establish some connections with Caffarelli's contraction theorem \cite{Caffarelli00}.
\end{abstract}

\section{Introduction}
Given two probability measures $\mu,\nu$ on $\R^d$, we recall that a \emph{transport plan} between $\mu$  and $\nu$ is a probability measure $\pi$ on $\R^d\times \R^d$ such that 
\[
\pi(A \times \R^d) = \mu(A) \qquad \text{and}\qquad \pi(\R^d\times B) = \nu(B),
\]
for all Borel sets $A,B$ of $\R^d.$ The set of all transport plans between $\mu$ and $\nu$ will be denoted by $C(\mu,\nu)$ in all the paper. It will be convenient to represent a transport plan $\pi \in C(\mu,\nu)$ in the following disintegrated form 
\begin{equation}\label{dis}
d\pi(x,y) = d\mu(x)dp_x(y),
\end{equation}
where $x\mapsto p_x$ is a ($\mu$-almost surely unique) probability kernel called a \emph{transport kernel}. A couple $(X,Y)$ of random variables such that $X\sim \mu$ and $Y\sim \nu$ is called a \emph{coupling} of $\mu$ and $\nu$. By a slight abuse of vocabulary, we also call couplings of $\mu$ and $\nu$ the transport plans of $C(\mu,\nu)$, which also explains the notation with letter $C$. In all the paper, $|\,\cdot\,|$ will denote the standard Euclidean norm on $\R^d$. The corresponding scalar product will be denoted by $x\cdot y$ or $\langle x,y\rangle$ when the expression of $x$ or $y$ is too long. For $k\geq 1$, we will denote by $\mathcal{P}_k(\R^d)$ the \emph{Wasserstein space} of order $k$, i.e the set of probability measures $\mu$ such that $\int |x|^k\,d\mu(x)<\infty$.

\subsection{Existence of structured couplings between probability measures} The question to construct couplings between probability measures having some nice structures or properties is natural both in probability theory and analysis. Let us recall two famous results guarantying the existence of couplings with rigid but very different structures, namely the theorems of Brenier and Strassen.

A fundamental result by Brenier \cite{Brenier87,Brenier91} (see \cite[\S 1.3]{McCG} for a discussion of the bibliography around this result) shows that if $\mu$ is say absolutely continuous with respect to Lebesgue measure, then there exists some convex function $\varphi : \R^d \to \R$ such that $\nabla \varphi$ pushes forward $\mu$ onto $\nu$. In other words, there exists a \emph{deterministic} coupling $\pi^\circ$ of the form 
\[
d\pi^\circ(x,y) = d\mu(x)\delta_{\nabla \varphi(x)}(y).
\]
Moreover, assuming in addition that $\mu$ and $\nu$ belong to $\mathcal{P}_2(\R^d)$, this $\pi^\circ$ is the unique optimal coupling in the Monge-Kantorovich transport problem for the quadratic cost:
\[
\iint |x-y|^2 \,d\pi^\circ(x,y) = \inf_{\pi \in C(\mu,\nu)} \iint |x-y|^2\,d\pi(x,y).
\]
This result has then been extended to various cost functions \cite{GMc96} and state spaces \cite{McCann01,Gigli11, AmRi, FiRi, FeUes, FaSh, Bertrand, GRS} and has had numerous applications in PDE, geometric analysis or probability theory, see \cite{AGS08, BGL14, Vil03,Vil09} and the references therein.

On the other hand, from a probabilistic point of view it is natural to investigate the existence of \emph{martingale couplings} between elements of $\mathcal{P}_1(\R^d)$, that is to say to look for couplings $\pi \in C(\mu,\nu)$ which correspond to the law of a martingale $(X_{t})_{t \in \{0;1\}}$. This martingale requirement means that the kernel $p$ appearing in \eqref{dis} satisfies
\[
\int y\,dp_x(y)=x,\quad\text{for } \mu \text{ almost every }x \in \R^d.
\]
As for the Brenier Theorem, it turns out that the existence of a martingale coupling between $\mu$ and $\nu$ is not automatic. Given $\mu,\nu \in \mathcal{P}_1(\R^d)$, a necessary and sufficient condition is given by Strassen Theorem \cite{Str65}: there exists a martingale coupling between $\mu$ and $\nu$ if and only if $\mu$ is dominated by $\nu$ in the convex order. Recall that if $\mu,\nu \in \mathcal{P}_1(\R^d) $, one says that $\mu$ is dominated by $\nu$ in the convex order if $\int f\,d\mu \leq \int f\,d\nu$ for all convex functions $f:\R^d\to\R.$ \footnote{Note that since a convex function $f$ on $\R^d$ is always bounded from below by some affine function, the integral of $f$ with respect to a probability measure with finite first moment always makes sense in $\R \cup \{+\infty\}$.} This is denoted as follows in the what follows: $\mu\leqc \nu$. Note that Strassen's result has been generalized at the time continuous process level by Kellerer \cite{Kel72,Kel73} (see also \cite{BJ2}): if a family $(\mu_t)_{t\geq 0}$ of elements of $\mathcal{P}_1(\R)$ is increasing for the convex order (a so-called \emph{peacocks} using the terminology of \cite{HiYo10}), then there exists a process $(X_t)_{t\geq0}$ which is together \emph{Markovian} and a martingale such that $X_t \sim \mu_t$ for all $t\geq0.$  See the monograph \cite{HPRY11} on peacocks, \cite{Low,HiRoYo14} for the approach by Lowther and \cite{Ju_seminaire,HiRo13,BouJui} for extensions.

\subsection{In between Brenier and Strassen}\label{sub:BaS}
Given two probability measures $\mu,\nu$ on $\R^d$ having their first moments finite, it is not always possible to go from $\mu$ to $\nu$ using a deterministic mapping given by the gradient of a convex function, nor to go from $\mu$ to $\nu$ using a martingale coupling. 
In this paper we will be interested in couplings obtained by composition of these two classical  transport methods, which as we will see in the next paragraph always exist.

More precisely, we will look for couplings $(X,Y)$ of $\mu$ and $\nu$ such that there exists a convex function $\varphi:\R^d\to \R$ of class $\mathcal{C}^1$ such that 
\begin{equation}\label{eq:Econd}
\E[Y|X] = \nabla \varphi(X),\qquad \text{a.s.}
\end{equation}
Introducing $X' = \nabla\varphi(X)$, we can see $(X,Y)$ as the initial and final states of a (time inhomogeneous) Markov chain 
$(X,X',Y)$ where the first transition step is deterministic and given by the gradient of a convex function and the second transition step is a martingale. As mentioned above, such couplings always exist: simply choosing  $X$ and $Y$ independent gives a trivial example corresponding to $\varphi(x)=x\cdot \E[Y]$, $x \in \R^d.$

The purpose of our main result (Theorem \ref{main-result} below) is therefore to distinguish some \emph{special} couplings $(X,Y)$ among those satisfying \eqref{eq:Econd} by showing that they are solutions of an optimal transport problem that we shall now present.


Let $\theta : \R^d \to \R^+$ be some convex function and $\mu,\nu \in \mathcal{P}_1(\R^d)$ ; using the terminology of \cite{GRST17}, the barycentric optimal transport cost between $\mu$ and $\nu$ is defined by
\begin{equation}\label{eq:barcost}
\overline{\mathcal{T}}_\theta(\nu|\mu) =\inf_{\pi \in C(\mu,\nu)} \int \theta\left(\int y \,dp_x(y)-x\right)\,d\mu(x),
\end{equation}
where $p$ is the kernel appearing in the disintegration formula \eqref{dis} for $\pi$. Notice that the bar over $\mathcal{T}_\theta$ in \eqref{eq:barcost} stands for the beginning of the word \emph{bar}ycentric.
In probabilistic notations, this optimal transport cost can be expressed as follows:
\[
\overline{\mathcal{T}}_\theta(\nu|\mu) = \inf \E\left[ \theta \left( \E[Y|X]-X\right)\right]
\]
where the infimum runs over the set of random vectors $(X,Y)$ such that $X \sim \mu$ and $Y \sim \nu.$ If $d\pi(x,y) = d\mu(x)dp_x(y) \in C(\mu,\nu)$ is such that $\overline{\mathcal{T}}_\theta(\nu|\mu) = \int \theta\left(\int y \,dp_x(y)-x\right)\,d\mu(x)$, we call it an \emph{optimal} transport plan (or coupling) from $\mu$ to $\nu$ (for the cost $\overline{\mathcal{T}}_\theta$). These barycentric costs are actually part of a more general family of transport costs, called ``weak transport costs" or ``generalized transport costs'', introduced under the first name by the first author together with Roberto, Samson and Tetali in \cite{GRST17} (see Section \ref{sec:weak} below for a more complete presentation). 

In this paper, we will mainly focus on the barycentric transport cost, denoted by $\overline{\mathcal{T}}_2$ in the remainder of the present text, in relation to $\overline{\mathcal{T}}_\theta$ defined in \eqref{eq:barcost}, as it is associated to the function $\theta(u) = |u|^2$, $u\in \R^d$. For all $\mu,\nu \in \mathcal{P}_1(\R^d)$
\[
\overline{\mathcal{T}}_2(\nu|\mu) = \inf_{\pi \in C(\mu,\nu)} \int \left|\int y \,dp_x(y)-x\right|^2\,d\mu(x).
\]
In \cite[Theorem 2.11]{GRST17}, the following Kantorovich type duality formula has been obtained:
\begin{thm}\label{them:duality}
If $\mu,\nu \in \mathcal{P}_1(\R^d)$ then
\[
\overline{\mathcal {T}}_\theta(\nu|\mu) = \sup_{f} \left\{ \int Q_\theta f\,d\mu - \int f \,d\nu\right\},
\]
where
\[
Q_\theta f(x) = \inf_{y\in \R^d}\{f(y) + \theta(y-x)\},\qquad x\in \R^d,
\]
and where the supremum runs over the set of all functions $f$ which are \emph{convex}, Lipschitz and bounded from below.
\end{thm}
Note that $Q_\theta(f)$ is the $c$-transform of $f$ related to the cost $c(x,y)=\theta(y-x)$, $x,y\in \R^d$, as defined in optimal transport theory (see for instance \cite[Definition 5.2]{Vil09}). The difference with respect to the usual duality formula for the global transport cost associated to $\theta$ (see \cite[Theorem 5.10]{Vil09}) is that we optimize over a class of convex functions. Several applications of this duality formula were already investigated in \cite{GRST17}, mainly in connection with transport-entropy inequalities and deviation inequalities for convex functions. In particular, let us mention that Strassen Theorem can be derived from this duality theorem, see \cite[Section 3]{GRST17}. Here we shall use this duality result to describe optimal transport plans for $\overline{\mathcal {T}}_2.$  

Before stating our main result, we need some preparation. First recall that the transport distance $W_2$ is defined for all probability measures $\mu,\nu \in \mathcal{P}_2(\R^d)$ by
\[
W_2^2(\mu,\nu) = \inf_{\pi \in C(\mu,\nu)} \iint |y-x|^2\,d\pi(x,y).
\]
Given a probability $\nu \in \mathcal{P}_1(\R^d)$, we denote by
\[
B_\nu = \{ \eta \in \mathcal{P}_1(\R^d) : \eta \leqc \nu\}
\]
the set of all probability measures which are dominated by $\nu$ in the convex order. This set $B_\nu$ is easily seen to be convex in the usual sense and the proof of Proposition \ref{prop:projection} will show that (generalized) $W_2$-geodesics with endpoints in $B_\nu$ are contained in $B_\nu$. 
\begin{prop}\label{prop:projection}
Let $\mu,\nu \in \mathcal{P}_2(\R^d)$. There exists a unique probability measure $\bar{\mu} \in B_\nu$ such that
\[
W_2(\bar{\mu},\mu) = \inf_{\eta \in B_\nu} W_2(\eta,\mu)= \overline{\mathcal{T}_2}(\nu|\mu).
\]
We call $\bar{\mu}$ the projection of $\mu$ on $B_\nu.$
\end{prop}
\begin{rem}
The identity 
\[
\overline{\mathcal{T}}_2(\nu|\mu) = \inf_{\eta \in B_\nu} W_2^2(\eta,\mu)
\]
of Proposition \ref{prop:projection} was already observed in \cite[Proposition 3.1]{GRSST15}  in dimension $1$ when $\mu$ \emph{has no atom} and then generalized to higher dimensions in \cite[Proposition 4.1]{Samsurv} when $\mu$ is \emph{absolutely continuous} with respect to Lebesgue. After a first version of this work has been released, we learned that Proposition \ref{prop:projection} has been independently obtained by Alfonsi, Corbetta and Jourdain in \cite{ACJ} in connections with the study of algorithms approximating the martingale transport problem.
\end{rem}
We recall that if $g : \R^d \to \R\cup\{+\infty\}$ is some convex function, the Legendre transform $g\mapsto g^*$ is defined by
\[
g^*(y) = \sup_{x\in \R^d}\{x\cdot y - g(x)\},\qquad y\in \R^d.
\]
We will keep on using notation ``$*$'' for the Legendre transform all along the paper. With these notions in hand, we can now state the main result of this paper which describes the set of all optimal couplings for $\overline{\mathcal{T}}_2$.
\begin{thm}\label{main-result}
Let $\mu, \nu \in \mathcal{P}_2(\R^d)$.
\begin{itemize}
\item[(a)] There exists some lower semi-continuous convex function $f^\circ:\R^d\to \R\cup\{+\infty\}$, which is integrable with respect to $\nu$ and such that
\begin{equation}\label{eq:attainment}
\overline{\mathcal{T}_2}(\nu|\mu)=\int Q_2f^\circ\,d\mu - \int f^\circ\,d\nu.
\end{equation}
In this equation $Q_2$ corresponds to $Q_\theta$, defined in Theorem \ref{them:duality}, where $\theta$ is the quadratic cost. It is given for any function $g:\R^d\to \R$ by
\[
Q_2g(x) = \inf_{y\in \R^d}\{ g(y) + |y-x|^2\},\qquad  x \in \R^d.
\]
\item[(b)] Let $h$ and $\varphi$ be the convex functions defined by
\[
h(x) = \frac{f^\circ(x)+|x|^2}{2}, \quad x\in \R^d \qquad\text{and}\qquad \varphi(y) = h^*(y),\quad y\in \R^d.
\]
The function $\varphi$ is $\mathcal{C}^1$-smooth on $\R^d$ and the map $\nabla \varphi$ is $1$-Lipschitz on $\R^d$. The projection $\bar{\mu}$ of $\mu$ on $B_\nu$ is such that 
\[
\bar{\mu} = \nabla \varphi_\# \mu.
\]

\item[(c)] The set of optimal transport plans for  $\overline{\mathcal{T}}_2(\nu|\mu)$ is non-empty. Moreover if $(X,Y)$ is a coupling of $\mu,\nu$ such that $\overline{\mathcal{T}}_2(\nu|\mu) = \E[|\E[Y|X]-X|^2]$, then $\E[Y|X]$ has law $\bar{\mu}$, $\E[Y|X] = \nabla \varphi (X)$ almost surely and  $(\E[Y|X],Y)$ is a martingale (this last point being of course always true).
\end{itemize}
\end{thm}

\begin{rem}\label{rem:deux_versions}\ 
\begin{itemize}
\item[-] The fact that $\varphi$ is convex and $\bar{\mu}=\nabla \varphi_\#\mu$ implies that $x\mapsto \nabla \varphi(x)$ provides the optimal transport plan $\pi=(\Id\times \nabla\varphi)_\#\mu \in C(\mu,\bar{\mu})$ for the quadratic cost. In other words, if $X$ has law $\mu$ then
\begin{equation}\label{eq:gradphiopti}
W_2^2(\bar{\mu},\mu) = \E[|\nabla \varphi (X) - X|^2].
\end{equation}
As in the classical optimal transport theory the conjugate potential $\psi= \varphi^* (= h)$ is such that $x\in \partial \psi(y)$, $\pi$-almost surely. Here $\partial \psi$ is the subgradient of $\psi$.
\item[-] We underline the fact that there is no assumption on the measure $\mu$ (as for instance the condition that it does not give mass to small sets, which is classical for Brenier transport, see \cite{Gigli11} for a minimal condition). As the potential function $\varphi$ is $\mathcal{C}^1$-smooth its gradient is well defined everywhere. Hence, the optimal transport map $x\mapsto \nabla\varphi(x)$ is pointwise well-defined and there is no ambiguity on the range of small sets.
\item[-] We stress that the optimal transport map from $\mu$ to $\bar{\mu}$ is continuous and even $1$-Lipschitz continuous, which is not automatic for quadratic optimal transport between arbitrary measures of $\mathcal{P}_2(\R^d)$. 
\item[-] In a previous version of this paper, Theorem \ref{main-result} appeared with the additional assumption that $\mu$ and $\nu$ are compactly supported. Soon after this first version has been released, a paper by Backhoff-Veraguas, Beiglb\"ock and Pammer \cite{BBP} proposed an improved version of Items (b) and (c) of our main result \emph{removing this compactness assumption}. Their approach is based on a clever combination of generalized cyclical monotonicity arguments and on the first compact version of our Theorem \ref{main-result}. In the present version of our paper, we improve our preceding proof of Item (a) of Theorem \ref{main-result} in order to remove also the compactness of supports assumption (the proof of Items (b) and (c) being essentially unchanged). The main difference between the previous proof and the present one is that the existence of the dual optimizer is obtained using some version of Ascoli Theorem adapted to convex functions instead of the usual version. 
To keep visible the incremental progresses around Theorem \ref{main-result}, we have treated separately the compact and the general cases in the proof Section \ref{sec:proofs}.
\end{itemize}
\end{rem}
Before presenting the organization of the paper, let us mention that the methods and results developed in this work are likely to extend to more general barycentric or even some classes of weak optimal transport problems, see \S \ref{sec:weak} for a presentation of the weak transport costs. We focus on the particular case of $\overline{\mathcal{T}}_2$ to avoid unnecessary generality and make the reading of this paper more fluid.

\subsection{Other results and organization of the paper} Let us now describe the content of the paper.

Section \ref{sec:contraction} adresses the question of characterizing equality cases between $W_2^2$ and $\overline{\mathcal{T}}_2$. The second main result of the paper, Theorem \ref{thm:equivalence}, states that $W_2^2(\nu,\mu) = \overline{\mathcal{T}}_2(\nu |\mu)$ \emph{if and only if} there exists some continuously differentiable convex function $\varphi$ such that $\nabla \varphi$ is $1$-Lipschitz and $\nu = \nabla \varphi _\# \mu$. In other words, this result shows that the Brenier map from $\mu$ to $\nu$ is a contraction, i.e a $1$-Lipschitz map, if and only if $\bar{\mu} = \nu.$ This observation is then compared to a classical theorem by Caffarelli  \cite{Caffarelli00} that states that probability measures with a log-concave density with respect to the standard Gaussian measure are contractions of it.

In Section \ref{sec:monotonicity}, we generalize the notion of $c-$monotonicity to the case of the transport cost $\overline{\mathcal{T}}_2$. We illustrate the power of this notion, by giving a new proof of Strassen Theorem for submartingales in dimension $1$.

In Section \ref{sec:Simplex}, we study a concrete example and describe the function $\nabla \varphi$ and the projected measure $\bar{\mu}$ when $\mu$ is some arbitrary probability measure on $\R^d$ (with finite second moment) and $\nu$ is a discrete probability measure concentrated on the vertices of a simplex in $\R^d.$

Section \ref{sec:examples} recalls some other examples and results appearing in the recent literature.

Finally, Section \ref{sec:proofs} contains the postponed proofs of Proposition \ref{prop:projection} and Theorem \ref{main-result} (for both the compact and general cases) and of some auxiliary lemmas.

%

\subsection{More about weak optimal transport costs}\label{sec:weak}
As mentioned at the end of \S \ref{sub:BaS}, the barycentric transport costs \eqref{eq:barcost} enter a more general family of transport costs, that we will now recall. 

Let $(\X,d)$ be a Polish space and $\mathcal{P}(\X)$ be the set of Borel probability measures on $\X.$
Given a cost function $c : \X \times \mathcal{P}(\X) \to \R^+$ and probability measures $\mu,\nu$ on $\X$ the optimal weak transport cost $\mathcal{T}_c(\nu|\mu)$ from $\mu$ to $\nu$ is defined by
\[
\mathcal{T}_c(\nu|\mu) = \inf_{\pi \in C(\mu,\nu)} \int c(x,p_x)\,d\mu(x),
\]
where $\pi$ and $p$ are related together by \eqref{dis}. Note that this definition makes sense under mild regularity assumptions on $c$, that we will not detail in this exposition.

On the one hand, when $c(x,p) = \int \rho(x,y)\,dp(y)$ for some measurable non-negative function $\rho$ on $\X^2$, one recovers the usual Monge-Kantorovich optimal transport cost 
\[
\mathcal{T}_\rho(\mu,\nu) = \inf_{\pi \in C(\mu,\nu)} \int \rho(x,y)\,d\pi(x,y).
\]
On the other hand, the optimal barycentric transport costs presented in \eqref{eq:barcost} corresponds to cost functions of the form $c(x,p)=\theta (\int y\,dp(y)-x)$. 

The first appearance of this form of transport cost goes back to the works of Marton \cite{Mar96a,Mar96b} and Talagrand \cite{Tal95,Tal96b} in connections with the concentration of measure phenomenon via the so-called transport-entropy inequalities (see \cite{Led01} for an introduction to concentration of measure phenomenon and \cite{GL10} for a survey on transport-entropy inequalities). For instance, Marton considers in \cite{Mar96b} the following simple cost function: 
\[
c_{\mathrm{Marton}}(x,p) = \left(\int \mathbf{1}_{x\neq y}\,dp(y)\right)^2,\qquad x \in \X,
\] 
and recovers a fundamental concentration of measure result by Talagrand \cite{Tal95} for product probability measures. After Marton, some other universal concentration of measures results by Talagrand were recovered or improved by Dembo \cite{Dem97} and Samson \cite{Sam00, Sam03, Sam07} using variants of Marton's cost $c_{\mathrm{Marton}}$. Motivated by questions related to concentration and curvature properties of discrete measures, the paper \cite{GRST17} introduces the general definition given above and studies some of the properties of this new class of transport costs. In particular Kantorovich type duality formulas are obtained \cite[Theorem 9.6]{GRST17} under the assumption that $c$ is convex with respect to the $p$ variable (and some additional mild regularity conditions). We refer to \cite{GRST14,GRSST15,Sam17, Shu18, FS18} for works directly connected to \cite{GRST17} and to \cite{Samsurv} for an up-to-date survey of applications of weak transport costs to concentration of measure.

Besides their many applications in the field of functional inequalities and concentration of measure, it turns out that weak transport costs are also interesting in themselves as a natural generalization of the transportation problem.  Indeed, the main interest of this definition is that it enables the introduction of additional constraints to the transfers of mass. For instance, a cost function of the form
\[
c(x,p) = \left\{\begin{array}{ll} \int \rho(x,y)\,dp(y)  & \text{if } \int y\,dp(y)=x   \\ +\infty & \text{otherwise}   \end{array}\right.
\]
where $\rho: \R^d\times \R^d$ is some non-negative measurable function gives back the notion of optimal transport with martingale constraint: if $\mu \leqc \nu$ (say compactly supported to avoid definition problem), then 
\[
\mathcal{T}_c(\nu|\mu) = \inf \{\E[\rho(X,Y)] : X\sim \mu, Y\sim \nu, (X,Y) \text{ is a martingale} \}.
\]
This optimal transport problem with martingale contraint has been thoroughly studied in \cite{BJ} for the dimension 1.  There, the martingale transport problem is studied for particular families of costs satisfying the cross derivative condition $\partial_x\partial_y^2 \rho<0$ giving rise to the left-curtain coupling (on this coupling, see also \cite{HT,Ju_ihp,HoNo}) and for the cost functions $\rho:(x,y)\mapsto \pm|y-x|$, generalizing results by Hobson and his coauthors, Neuberger and Klimmek \cite{HN,HK}, respectively. The supermartingale problem  in dimension 1 with cross-derivative condition is studied in \cite{NS}. In higher dimension $\rho:(x,y)\mapsto \pm\|y-x\|$ is studied in \cite{GKL,Lim2014}, more general costs are considered in \cite{March2018a}. We refer to \cite{ABC18} for existence results in the dual optimal transport problem with martingale contraints making use of the formulation in terms of weak costs as above.

As observed by Conforti \cite{conforti17} and Conforti and Ripani \cite{CR18} the class of weak transport costs also entails the notion of entropic costs related to the Schrödinger problem (see the nice survey by C. Léonard \cite{Leo14}). 
We also refer to the recent articles by Alibert, Bouchitt\'e and Champion \cite{ABC18} (see also Section \ref{sec:examples}) and Bowles and Ghoussoub \cite{BG18} for further developments and examples.

\section{Link with Caffarelli's contraction theorem}\label{sec:contraction}

In this section we admit the main theorem of this paper, Theorem \ref{main-result}, and derive a link with the celebrated Caffarelli's contraction theorem of \cite{Caffarelli00}.
Precisely we investigate the question of determining in which case 
\[
W_2^2(\nu,\mu) = \overline{\mathcal{T}}_2(\nu|\mu)
\]
and we show how this question is related (when $\mu$ is the standard Gaussian) to Caffarelli's result. The following result is the second main contribution of the paper:
\begin{thm}\label{thm:equivalence}
Let $\mu,\nu \in \mathcal{P}_2(\R^d)$ with $\R^d$ equipped with the standard Euclidean norm. The following statements are equivalent:
\begin{itemize}
\item[(a)] There exists a convex function $\varphi : \R^d\to \R$ of class $\mathcal{C}^1$ such that $\nabla \varphi$ is $1$-Lipschitz such that $\nu = \nabla \varphi_\# \mu.$
\item[(b)] The projection $\bar{\mu}$ of $\mu$ on the set $B_\nu = \{ \eta \in \mathcal{P}_1(\R^d)  : \eta \preceq_c \nu\}$ is $\nu.$
\item[(c)] It holds $W_2^2(\nu,\mu) = \overline{\mathcal{T}}_2(\nu|\mu)$.
\end{itemize}
\end{thm}
In dimension $1$, the equivalence between (c) and (a) has been obtained by Shu in a slightly different form in \cite{Shu}.

Let us denote by $\gamma_d$ the standard Gaussian measure on $\R^d$ and recall the statement of Caffarelli's contraction theorem.
\begin{thm}[Caffarelli's contraction theorem \cite{Caffarelli00}]\label{thm:Caffarelli}
If $\nu$ is a probability measure with a density with respect to $\gamma_d$ of the form $e^{-V}$, with $V : \R^d\to \R\cup \{\infty\}$ a convex function, then there exists a convex function $\varphi$ of class $\mathcal{C}^1$ on $\R^d$ such that the Brenier transport map $\nabla \varphi$ from $\gamma_d$ to $\nu$ is $1$-Lipschitz.
\end{thm}
Theorem \ref{thm:Caffarelli} is an important result with numerous applications in the field of functional inequalities \cite{Harge01,Cordero02, Milman18}. Note that the assumption on $\nu$ can equivalently be formulated as: $\nu$ has density $e^{-U}$ with respect to the Lebesgue measure where $x\mapsto U(x)-|x|^2/2\in \R\cup \{\infty\}$ is convex. Such measures are called \emph{$1$-uniformly log-concave} in the literature.

 The following corollary is an immediate consequence of Theorems \ref{thm:equivalence} and \ref{thm:Caffarelli}.
\begin{cor}
If $\nu$ is a probability measure with a density with respect to $\gamma_d$ of the form $e^{-V}$, with $V : \R^d\to \R\cup \{+\infty\}$ a convex function, then the projection $\overline{\gamma_d}$ of $\gamma_d$ on $B_\nu$ is equal to $\nu$.
\end{cor}

\begin{rem}Let $d\nu = e^{-V}\,d\gamma_d$ with $V$ convex on $\R^d.$
\begin{itemize}
\item[-] According to Theorem \ref{thm:equivalence} the conclusion of Theorem \ref{thm:Caffarelli} is logically equivalent to the statement $\overline{\gamma_d} = \nu.$ It would be very interesting to prove directly that $\overline{\gamma_d} = \nu$ since this would give an alternative proof Theorem \ref{thm:Caffarelli}. This question will be considered elsewhere.
\item[-] To complete the picture, let us mention a nice result of Harg\'e \cite{Harge01} that states that if in addition $\int x\,d\nu(x) = 0$, then $\nu \preceq_c \gamma_d$.
\end{itemize}
\end{rem}

Now let us turn to the proof of Theorem \ref{thm:equivalence}. We will need the following classical lemma (see \cite[Theorem E 4.2.1]{HUL01}) whose proof is recalled in Section \ref{sec:proofs} for the sake of completeness.
\begin{lem}\label{lem:contraction}
Let $g:\R^d\to \R\cup\{+\infty\}$ be a lower semi-continuous convex function such that $g(x_o)<+\infty$ for some $x_o \in \R^d$ and recall $g^*(y) = \sup_{x\in \R^d}\{x\cdot y - g(x)\}$, $y\in \R^d$. \linebreak
The following are equivalent: 
\begin{itemize}
\item[(a)] The function $x\mapsto g(x) - \frac{|x|^2}{2}$ is convex on $\R^d$.
\item[(b)] The function $y\mapsto  \frac{|y|^2}{2} - g^*(y)$ is convex and finite valued on $\R^d$.
\item[(c)] The function  $g^*$ is of class $\mathcal{C}^1$ and $\nabla g^*$ is $1$-Lipschitz on $\R^d.$
\end{itemize}
\end{lem}
Notice that (a) quantitatively states that $g$ is more convex than $x\mapsto|x|^2/2$ while (b) states that $g^*$ is less convex than $x\mapsto |x|^2/2$. Now let us prove Theorem \ref{thm:equivalence}.
\proof[Proof of Theorem \ref{thm:equivalence}]
Let us show that (a) implies (b). Define $h =\varphi^*$ and $f = 2\varphi^* - |x|^2$. According to Lemma \ref{lem:contraction} (the implication $(c)\Rightarrow (a)$ applied with $g= \varphi^*$), we have that $f$ is convex. An easy calculation shows that $Q_2f(x) = |x|^2 - 2\varphi(x)$, for every $x \in \R^d$, and moreover it holds $\varphi^*(\nabla \varphi(x)) = x\cdot \nabla \varphi(x) - \varphi(x)$, $x \in \R^d.$ This together with the fact that $\nabla \varphi$ sends $\mu$ onto $\nu$ yields that
\begin{align*}
W_2^2(\nu,\mu) & \leq \int |\nabla \varphi(x)-x|^2\,d\mu(x)\\
& = \int |x|^2\,d\mu(x) + \int |y|^2\,d\nu(y) - 2 \int x\cdot \nabla \varphi(x)\,d\mu(x)\\
& =  \int |x|^2\,d\mu(x) + \int |y|^2\,d\nu(y) - 2 \int \varphi(x) + \varphi^*(\nabla \varphi(x))\,d\mu(x)\\
& = \int Q_2f\,d\mu - \int f \,d\nu\\
& \leq \overline{\mathcal{T}}_2(\nu|\mu)\\
& \leq W_2^2(\nu,\mu),
\end{align*}
where the second inequality comes from the convexity of $f$ and the duality formula for $\overline{\mathcal{T}}_2$, see Theorem  \ref{main-result} (a). 
Note that, according to Lemma \ref{lem:contraction} (the implication $(c)\Rightarrow (b)$ with $g=\varphi^*$) we have that $y \mapsto |y|^2/2 - \varphi(y)$ is convex. Therefore it is bounded from below by some affine function and so $\varphi$ is bounded from above by some quadratic function, which shows that $\varphi$ is integrable with respect to $\mu$. Since $x\mapsto \varphi(x) + \varphi^*(\nabla \varphi(x))$ is integrable with respect to $\mu$, one concludes that  $\varphi^*(\nabla \varphi(x))$ is also integrable with respect to $\mu$, which enables to split the integral at the fourth line. This also shows that $f$ is integrable with respect to $\nu.$

Therefore, the function $f$ is optimal for the dual problem and so, according to Theorem \ref{main-result}, $\nu=\nabla h^*_\#\mu = \nabla \varphi_\#\mu = \bar{\mu}$, which shows (b). Now according to Theorem \ref{main-result}, $W_2^2(\bar{\mu},\mu) = \overline{\mathcal{T}}_2(\nu|\mu)$ so (b) implies (c). On the other hand, if $W_2^2(\nu,\mu) = \overline{\mathcal{T}}_2(\nu|\mu)$, then according to Proposition \ref{prop:projection} and Theorem \ref{main-result}, it holds $\bar{\mu} = \mu$, so (c) implies (b) as well. Finally, according to Theorem \ref{main-result}, there always exist a convex function $\varphi$ of class $\mathcal{C}^1$ on $\R^d$ such that $\nabla \varphi$ is $1$-Lipschitz and  $\bar{\mu} = \nabla \varphi_\# \mu$. So (b) implies (a), which completes the proof.
\endproof

\section{A monotonicity theorem}\label{sec:monotonicity}
In this section we still admit the main theorem of this paper, Theorem \ref{main-result}. We generalize the notion of $c$-monotonicity, which plays an important role in optimal transport theory \cite{RR90, GMc96}, to cost functions $c$ defined on $\R^d \times \mathcal{P}(\R^d)$ and we use it to formulate a necessary condition of optimality for transport plans. Then we give a new proof of Strassen theorem on the existence of submartingale couplings with given marginals in dimension $1$ using this $c$-monotonicity criterium.

\subsection{A necessary condition of optimality} 
The following definition generalizes the notion of $c$-monotonicity ; we refer to \cite{RR98,Vil03,Vil09} and the references therein for a complete account on the subject. In what follows, we denote by $\mathcal{P}_1(\R^d)$ the set of probability measures on $\R^d$ having finite first moment.
\begin{defi}
Let $c: \R^d \times \mathcal{P}_1(\R^d) \to \R^+$ be a cost function. 
We will say that a set $\Gamma \subset \R^d \times \mathcal{P}_1(\R^d)$ is $c$-monotone if for any $N \geq 1$ and points $(x_i,p_i) \in \Gamma$ and probability measures $q_i \in \mathcal{P}_1(\R^d)$, $i \in \{1,\ldots, N\}$ such that $\sum_{i=1}^N p_i = \sum_{i=1}^N q_i $, it holds 
\[
\sum_{i=1}^N c(x_i,p_i) \leq \sum_{i=1}^N c(x_i,q_i).
\]
\end{defi}
We now formulate a necessary condition for optimality of a coupling for the transport cost $\overline{\mathcal{T}_2}$ which we recall is associated to the cost function $c_2$ defined by
\[
c_2(x,p) = \left|\int y\,dp(y)-x\right|^2,\qquad x \in \R^d, \qquad p \in \mathcal{P}_1(\R^d).
\]
\begin{thm}\label{thm:monoto}
Let $\mu,\nu \in \mathcal{P}_2(\R^d)$. There exists a set $\Gamma \subset \R^d \times \mathcal{P}_1(\R^d)$ which is monotone with respect to the cost function $c_2$ and such that if $d\pi(x,y) = d\mu(x)dp_x(y)$ is an optimal coupling for the cost $\overline{\mathcal{T}}_2(\nu|\mu)$, it holds 
\[
\mu\left( \{x\in \R^d : (x,p_x) \in \Gamma\}\right)=1.
\]
\end{thm}
\begin{rem}
The paper \cite{BBP} contains a general version of Theorem \ref{thm:monoto} relying on a completely different proof (see \cite[Theorem 3.4]{BBP}). It is also shown in \cite[Theorem 3.6]{BBP} that, under mild assumptions on the cost function (verified for instance by the cost $c_2$), $c-$monotonicity is also a sufficient condition of optimality.
\end{rem}
\proof Recall the duality formula 
\[
\overline{\mathcal{T}}_2(\nu|\mu) = \sup_{f}\left\{ \int Q_2f(x)\,d\mu(x) - \int f(y)\,d\nu(y)\right\},
\]
where the infimum is running over the set of convex functions bounded from below and $Q_2f(x) = \inf_{y\in \R^d}\{f(y) + |y-x|^2\}$, $x \in\R^d$.
Note that if $f$ is convex, then
\[
Q_2f(x) = \inf_{p \in \mathcal{P}_1(\R^d)} \left\{c_2(x,p)+ \int f\,dp\right\},\qquad x \in \R^d.
\]
Indeed, the inequality $\geq$ is obtained by taking $p=\delta_y$. Moreover, since $f$ is convex, it holds
\[
\int f(y)\,dp(y) + c_2(x,p) \geq f\left(\int y\,dp(y)\right) +\left |\int y\,dp(y)-x\right|^2 \geq Q_2f(x)
\]
so taking the infimum over $p$ gives the converse inequality.
Let $f^\circ$ be such that $\overline{\mathcal{T}}_2(\nu|\mu) =  \int Q_2f^\circ(x)\,dp(x) - \int f^\circ(y)\,d\nu(y)$ (which exists according to Theorem \ref{main-result}). Then, it holds
\begin{align*}
\overline{\mathcal{T}}_2(\nu|\mu) &=  \int Q_2f^\circ(x)\,d\mu(x) - \int f^\circ(y)\,d\nu(y)\\
& \leq \int \left( \int f^\circ(y)\,dp_x(y) + c_2(x,p_x) \right)\,d\mu(x) - \int f^\circ(y)\,d\nu(y)\\ 
&= \int c_2(x,p_x)\,d\mu(x) = \overline{\mathcal{T}}_2(\nu|\mu).
\end{align*}
Therefore, $Q_2f^\circ(x) = \int f^\circ(y)\,dp_x(y) + c_2(x,p_x) $ for $\mu$-almost all $x \in \R^d.$
So denoting by $\Gamma$ the set of couples $(x,p)$ so that $Qf^\circ(x) = \int f^\circ\,dp + c_2(x,p)$, we have that $\mu(\{x \in \R^d : (x,p_x) \in \Gamma\})=1.$ Now, let us check that $\Gamma$ is monotone with respect to the cost function $c_2$. Take a family of points $(x_i,p_i) \in \Gamma$ and probability measures $q_i \in \mathcal{P}_1(\R^d)$, $i \in \{1,\ldots, N\}$ such that $\sum_{i=1}^N p_i = \sum_{i=1}^N q_i $. Then, it holds
\[
Q_2f^\circ(x_i) = c_2(x_i,p_i)+ \int f^\circ\,dp_i,\qquad \forall i \in \{1,\ldots, N\}
\]
and on the other hand
\[
Q_2f^\circ(x_i) \leq  c_2(x_i,q_i)+\int f^\circ\,dq_i , \qquad \forall i \in \{1,\ldots, N\}.
\]
Summing these inequalities, and using the fact that $\sum_{i=1}^N p_i = \sum_{i=1}^N q_i $ gives immediately $\sum_{i=1}^N c(x_i,p_i) \leq \sum_{i=1}^N c(x_i,q_i).$ \endproof

\begin{rem}
Using the notation of Theorem \ref{main-result}, let us show how, with Theorem \ref{thm:monoto}, one can recover the fact proved in Theorem \ref{main-result} that $\nabla \varphi$ is $1$-Lipschitz on the support of $\mu$. Let $d\pi(x,y) = d\mu(x)dp_x(y)$ be an optimal transport plan from $\mu$ to $\nu$ and $\Gamma$ be a $c_2$-monotone set such that $(x,p_x) \in \Gamma$ for $\mu$ almost every $x \in \R^d.$  If $(x,p_x)$ and $(y,p_y)$ are elements of $\Gamma$, then comparing this pair with $(x,(1-\eps)p_x+\eps p_y)$ and $(y,(1-\eps)p_y+\eps p_x)$ where $\varepsilon \in [0,1]$  yields
\[
|x-\bar{x}|^2+|y-\bar{y}|^2\leq |x-[(1-\eps)\bar{x}+\eps \bar{y}]|^2+|y-[(1-\eps)\bar{y}+\eps \bar{x}]|^2,
\]
where $\bar{x}=\int z \,dp_x(z)$ and $\bar{y}=\int z \,dp_z(z)$. 
Noting that this is an equality for $\eps=0$ and taking the right derivative at $\eps=0_+$, we thus obtain
\[2\langle \bar{x}-x,\bar{y}-\bar{x}\rangle-2\langle \bar{y}-y,\bar{y}-\bar{x}\rangle=2\langle y-x,\bar{y}-\bar{x}\rangle-2|\bar{y}-\bar{x}|^2\geq 0,
\]
hence $|\bar{y}-\bar{x}|\leq |y-x|$. Since $\bar{x} = \nabla \varphi(x)$ and $\bar{y}=\nabla \varphi(y)$ this proves that $\nabla \varphi$ is $1$-Lipschitz on the support of $\mu$.
\end{rem}

\subsection{An example related to the increasing convex order}

\begin{defi}
Let $\mu,\nu$ be two probability measures on $\R$  ; $\mu$ is dominated by $\nu$ for the \emph{stochastic order}, denoted by $\mu\leqs \nu$, if $\int f\,d\mu \leq \int f\,d\nu$ for all bounded increasing functions $f:\R\to\R$. Assuming that $\mu,\nu$ have finite first moments, $\mu$ is dominated by $\nu$ for the \emph{increasing convex order}, denoted by $\mu \leqcs \nu$, if $\int f \,d\mu\leq \int f \,d\nu$ for every increasing convex function $f:\R\to \R$. 
\end{defi}
Notice that if $\int x\,d\mu(x)=\int y\,d\nu(y)$ and $\mu\leqcs \nu$, then it is easy to see that the inequality $\int f\,d\mu\leq \int f\,d\nu$ can be extended, first to all convex functions $f$ with finite asymptotic slopes at $- \infty$, and second, by using the monotone convergence theorem, to all convex functions. Therefore, for probability measures with finite first moments, $\mu\leqc \nu$ is equivalent to ``$\mu\leqcs \nu$ and $\int x\,d\mu(x)=\int x\,d\nu(x)$''.

In this section we prove that on $\R$ under the assumption $\mu\leqcs \nu$, our Theorem \ref{main-result} on the quadratic barycentric transport problem can be completed with new informations.

\begin{thm}\label{csto}
Let $\mu, \nu \in \mathcal{P}_2(\R)$ be such that $\mu\leqcs \nu$. Then, with the notations of  Proposition \ref{prop:projection} and Theorem \ref{main-result}, the probability measure $\bar{\mu}$ satisfies $\mu \leqs \bar{\mu}$ and any optimal coupling $(X,Y)$ satisfies $X\leq \E[Y|X]$ almost surely.
\end{thm}
In the terminology of martingale transport, $(X,Y)$ is a submartingale coupling and its law a submartingale transport plan.

Before we prove Theorem \ref{csto}, note that conversely to it, if $(X,Y)$ is a coupling of $\mu$ and $\nu$ such that $X \leq \E[Y|X]$ almost surely, then for every increasing and convex function $f:\R\to \R$, it holds 
\[
\E[f(Y)]\geq \E[\E[f(Y)|X]]\geq \E[f(\E[Y|X])]\geq \E[f(X)].
\]
and so $\mu \leqcs \nu$.
We therefore recover the classical result by Strassen \cite[Theorem 9]{Str65} on submartingale couplings (under the additional assumption that the measures have finite second moments): $\mu\leqcs \nu$ if and only if there exists $X\sim \mu$ and $Y\sim \nu$ in the same probability space with $X \leq \E[Y|X]$ almost surely.

\begin{rem}
Theorem \ref{csto} also gives back Strassen Theorem in the very specific case of the real line and measures with finite second moments. Let us prove both implications. The fact that $\E[Y|X]=X$ implies $\mu\leqc \nu$ is an obvious consequence of the conditional Jensen inequality. Conversely, if $\mu\leqc \nu$ first notice  $\mu\leqcs \nu$. Therefore according to Theorem \ref{csto} there exists $X\sim \mu$ and $Y\sim \nu$ in the same probability space with $\E[Y|X]\geq X$. Moreover, $\mu\leqc \nu$ implies $\int x \,d\mu(x)=\int x \,d\nu(x)$, i.e. $\E[X]=\E[Y]$. Hence $\E[Y|X]$ and $X$ are ordered random variables with the same expectation. They are therefore almost surely equal, which proves the converse implication of Strassen Theorem. For another elementary proof, see \cite[Section 2]{BJ}
\end{rem}

\begin{proof}[Proof of Theorem \ref{csto}]
We denote a solution of the quadratic barycentric problem by $d\pi(x,y)=d\mu(x)dp_x(y)$. Let $(X,Y)$ such that $(X,Y)\sim \pi$ and denote $\E[Y|X]$ by $\bar{X}$. According to Theorem \ref{main-result}, $\pi$ may not be uniquely determined but the law of $\bar{X}$ is uniquely determined and denoted by $\bar{\mu}$.
The inequality $X \leq \E[Y|X]$ almost surely immediately implies that $\mu \leqs \bar{\mu}$.
Now to show that $X \leq \E[Y|X]$ almost surely amounts to show that $\int y \,dp_x(y) \geq x$ for $\mu$ almost all $x \in \R.$ The rest of the proof is devoted to this question.

According to Theorem \ref{thm:monoto}, there exists $\Gamma\subset \R\times \mathcal{P}_1(\R)$ such that, $\mu( \{x \in \R : (x,p_x) \in \Gamma\})=1$ and such that if $(x,p)\in \Gamma$ and $(x',p')\in \Gamma$, 
then for all probability measures $q,q'$ such that $p+p' = q+q'$, it holds 
\begin{equation}\label{eq:mon}
\left|\int y\,dp(y)-x\right|^2+ \left| \int y\,dp'(y)-x'\right|^2 \leq \left| \int y\,dq(y)-x\right|^2+ \left| \int y\,dq'(y)-x'\right|^2.
\end{equation}
Let us show that if $x,x' \in \R$ are such that $(x,p_x) \in \Gamma$ and 
$(x',p_{x'}) \in \Gamma$, then for all $a \in \mathrm{Support}(p_x)$ and $b\in \mathrm{Support}(p_{x'})$, it holds
\begin{equation}\label{eq:mon2}
\left[\left( \int y\,dp_x(y)-x\right)-\left( \int y\,dp_{x'}(y)-x'\right) \right](b-a)\geq 0.
\end{equation}
Let $a \in \mathrm{Support}(p_x)$ and $b\in \mathrm{Support}(p_{x'})$ and $\varepsilon>0$ and define 
\[
r_a = p_x(\,\cdot\, | [a-\varepsilon,a+\varepsilon])\qquad\text{and} \qquad r_b = p_{x'}(\,\cdot\, | [b-\varepsilon,b+\varepsilon]).
\]
For all $t>0$, define
\[
q_x^t =p_x + t(r_b-r_a) \qquad\text{and}\qquad q_{x'}^t= p_{x'}+t(r_a-r_b)
\]
and note that if $t$ is small enough, $q_x^t$ and $q_{x'}^t$ are probability measures such that $q_x^t+q_{x'}^t = p_x+p_{x'}$. Applying \eqref{eq:mon} and letting $t\to 0$ gives
\[
\left[\left( \int y\,dp_x(y)-x\right)-\left( \int y\,dp_{x'}(y)-x'\right) \right]\left(\int y\,dr_b(y) - \int y\,dr_a(y)\right)\geq 0.
\]
Finally, letting $\varepsilon\to 0$ yields to \eqref{eq:mon2}.

Define $\Gamma^1=\{x \in \R : (x,p_x) \in \Gamma\}$ and for any interval $I\subset \R$, let $A_I$ the set of points $x\in \Gamma^1$ such that $\int y \,dp_x(y)-x\in I$. We aim at proving that $\mu(A_{]-\infty,0[})=0$. 

Striving for a contradiction suppose that $\mu(A_{]-\infty,0[})>0$. Our assumption $\mu\leqcs \nu$ implies $\int x \,d\mu(x)\leq\int y \,d\nu(y)$ so that $\int_{\Gamma^1} [(\int y \,dp_x(y))-x]d\mu(x)\geq 0$. It follows that $\mu(A_{[0,+\infty[})>0$. Henceforth $A_{]-\infty,0[}$ and $A_{[0,+\infty[}$ are not empty. 
Several configurations may be considered. 

\begin{itemize}
\item[-] In the first case we assume that there exists $x' \in A_{]-\infty,0[}$ and $x \in A_{[0,+\infty[}$ such that $x'\leq x$.
One thus has $\int y\,dp_{x'}(y)<x'\leq x\leq \int y \,d p_x(y)$. So in this case, one can find $a \in \mathrm{Support}(p_x)$ and $b \in \mathrm{Support}(p_{x'})$ such that $a >b$. Applying \eqref{eq:mon2} with these $a,b$ provides a contradiction.
\item[-] We can now assume that $A_{[0,+\infty[}$ has a supremum smaller than or equal to the infimum of $A_{]-\infty,0[}$. If  there exists  $a\in \mathrm{Support}(p_x)$ and $b \in \mathrm{Support}(p_{x'})$ such that $a >b$, we get the same contradiction as before.
\item[-] We are finally reduced to the case where $A_{[0,+\infty[}$ has a supremum smaller than or equal to the infimum of $A_{]-\infty,0[}$ and moreover every $p_x$ with $x\in A_{[0,+\infty[}$ has the essential supremum smaller than or equal to the essential infimum of every $p_{x'}$ with $x'\in A_{]-\infty,0[}$. Let $m\in \R$ be between $\sup_{x \in A_{[0,+\infty[}}[\sup\mathrm{Support}(p_x)]$ and $\inf_{x'\in A_{]-\infty,0[}}[\inf\mathrm{Support}(p_{x'})]$. Integration of $x\mapsto (x-m)_+$ provides a contradiction to $\mu\leqcs \nu$. Indeed, one the one hand, it holds
\[
\int [x-m]_+ \,d\mu(x) \geq \int [x-m]_+ \mathbf{1}_{A_{]-\infty,0[}}(x)\,d\mu(x)
\]
(with actually equality).
On the other hand,
\begin{align*}
\int [y-m]_+ \,d\nu(y) &= \int \left(\int [y-m]_+ \,dp_x(y)\right) \,d\mu(x)\\
& \overset{(i)}{=} \int \left(\int [y-m]_+ \,dp_x(y)\right)\mathbf{1}_{A_{]-\infty,0[}}(x) \,d\mu(x)\\
& \overset{(ii)}{=} \int \left(\int y-m \,dp_x(y)\right)\mathbf{1}_{A_{]-\infty,0[}}(x) \,d\mu(x)\\
& \overset{(iii)}{<} \int \left(x-m \right)\mathbf{1}_{A_{]-\infty,0[}}(x) \,d\mu(x)\\
& \overset{(iv)}{=} \int [x-m]_+ \mathbf{1}_{A_{]-\infty,0[}}(x) \,d\mu(x),
\end{align*}
where 
\begin{itemize}
\item $(i)$ comes from the fact that for any $x \in A_{[0,+\infty[}$, $\mathrm{Support}(p_x) \subset ]-\infty,m]$,
\item $(ii)$ comes from the fact that for any $x \in A_{]-\infty,0[}$, $\mathrm{Support}(p_x) \subset [m,+\infty[$,
\item $(iii)$ comes from the definition of $A_{]-\infty,0[}$ and the fact that if $f<g$ on a set $A$ such that $\mu(A)>0$, then $\int_A f\,d\mu < \int_A g\,d\mu$,
\item $(iv)$ follows from the inequality $x > \int y\,dp_x(y) \geq m$, for all $x \in A_{]-\infty,0[}.$
\end{itemize}
\end{itemize}

In any case there is a contradiction to the fact that $\mu(A_{]-\infty,0[})>0$ and this completes the proof.
%
\end{proof}

\begin{rem}\ 
\begin{enumerate}
\item Optimal transport plans for the weak transport problem are not uniquely determined. As stated in Theorem \ref{main-result}, $\bar{\mu}$ and the transport from $\mu$ to $\bar{\mu}$ are uniquely determined but the martingale transport plan from $\bar{\mu}$ to $\nu$ is completely free. It is precisely the goal of Martingale Optimal Transport to determine special martingale transport plans from $\bar{\mu}$ to $\nu$.
\item This remark holds in particular for the dimension 1. Namely, in Theorem \ref{csto} we did not defined a submartingale coupling but only the non-decreasing/increasing part of its Doob decomposition. In relation with this decomposition, the paper \cite{NS} proposes two supermartingale transport problems and describes their solutions, extending the theory on curtain couplings initiated in \cite{BJ}. In particular the two optimal supermartingale couplings respectively corresponding to the two problems coincide with the curtain coupling when $\mu\leqc \nu$. However, transforming the supermartingale problem in submartingale problem, we stress that $\bar{X}=\E[Y|X]$ and its coupling with $X$ is different from ours as we show in this example: if $X$ is the uniform measure on $[-1,0]$ and $Y$ is uniform on $[0,3]$, our coupling is $\E[Y|X]=\bar{X}=X+2$ (it is a translation see Section \ref{sec:Simplex}). The first supermartingale coupling by Nutz and Stebegg (reversed in order to be a submartingale) gives $Y=\E[Y|X]=3(X+1)$. The second one provides $Y=\E[Y|X]=-3X$.
\end{enumerate}
\end{rem}

\section{Example of $\nu$ concentrated on the vertices of a simplex}\label{sec:Simplex}

In this section we prove the following theorem, mainly by geometric means.

\begin{thm}
Let $\mu$ be a compactly supported probability measure on $\R^d$ and $\nu$ an atomic measure whose support is a simplex $\{y_0,\ldots, y_k\}$ of $\R^d$ with $k\leq d$. Denote the convex hull of $\{y_0,\ldots, y_k\}$ by $\Delta$. Then there exists $v\in \R^d$ such that the map $T$ defined by
\[T: \R^d \to \R^d : x\mapsto T(x)=\mathrm{proj}_\Delta(x+v)\]
is such that $\bar{\mu}=T_\#\mu$ (with the notations of Proposition \ref{prop:projection} and Theorem \ref{main-result}), where $\mathrm{proj}_\Delta$ denotes the orthogonal projection on the closed convex set $\Delta.$ 
\end{thm}

\begin{rem}\ 
\begin{itemize}
\item[-] To be consistent with Theorem \ref{main-result}, note that the map $T$ given above can be written as $T = \nabla \varphi$, where $\varphi:\R^d \to \R$ is the convex function of class $\mathcal{C}^1$ defined by
\[
\varphi(x) = \frac{|x+v|^2}{2} - \frac{1}{2}d(x+v, \Delta)^2,\qquad x \in \R^d,
\]
where $d(z,\Delta) = \inf_{y \in \Delta} |z-y|$.
\item[-] The assumption that $\mu$ is compactly supported could be easily relaxed into the assumption that $\mu$ admits a moment of order $2$ finite.
\end{itemize}
\end{rem}

\proof
Let $\mu$ be a probability with compact support and $\nu$ an atomic measure with support the simplex $\{y_0,\ldots,y_k\}$ of $\R^d$ with $k\leq d$. We denote by $\bar{\mu}$ the projection of $\mu$ on $B_\nu= \{\eta \in \mathcal{P}(\R^d) : \eta \leqc \nu\}$ given by Proposition \ref{prop:projection} and Theorem \ref{main-result}. The assumption that the support of $\nu$ is a simplex is in order for the following property: for any point $y$ in the convex hull $\Delta=\mathrm{Conv}(y_0,\ldots,y_k)$ the barycentric coordinates $(\lambda_0,\ldots,\lambda_k)$, defined by $\sum_{i=0}^k \lambda_i=1$ and $\sum_{i=0}^k \lambda_i y_i=y$, are uniquely determined. Moreover, notice that all the coordinates are non-negative. If $v \in \R^d$, we will denote by $\mu_v$ the translation of $\mu$ by the vector $v$, \textit{i.e} $\mu_v = \mathrm{Law}(X+v)$, with $X \sim \mu.$ The four following properties will permit us to determine $\bar{\mu}$ and find the optimal coupling between $\mu$ and $\bar{\mu}$.

\begin{itemize}
\item[(a)]  \label{pt:trans} Let $\eta$ be a probability measure with a finite moment of order $2$. A coupling $(X,Y)$ is an optimal coupling for $W_2$ between $\mu$ and $\eta$ if and only if $(X+v,Y)$ is an optimal coupling for $W_2$  between $\mu_v$ and $\eta$.
Moreover,
\[
W_2^2(\mu_v,\eta)=W_2^2(\mu,\eta)+|v|^2+2\left\langle \int x\,d\mu-\int x\,d\eta, v\right\rangle,
\]
\item[(b)] \label{pt:ordre} As $\bar{\mu}\leqc \nu$ the measure $\bar{\mu}$ is concentrated on $\Delta$ and has the same barycenter as $\nu$. Conversely any measure $\eta$ concentrated on $\Delta$ and with the same barycenter as $\nu$ satisfies $\eta\leqc \nu$.
\item[(c)] \label{pt:proj} Among the measures concentrated on $\Delta$, the one that minimises the quadratic cost $W_2$ with respect to a given probability measure $\mu'$ having a finite moment of order $2$ is $(\mathrm{proj}_\Delta)_\#\mu'$.
\item[(d)] \label{pt:topo} There exists $v\in \R^d$ such that $(\mathrm{proj}_\Delta)_\#\mu_v$ has the same barycenter as $\nu$.
\end{itemize}

Before we prove the different four points, let us finish the proof. According to Item (b), all the elements of $B_\nu$ have the same barycenter as $\nu$. Therefore, applying Item (a), we have for any $\eta \in B_\nu$
\[
W_2^2(\mu_v,\eta)=W_2^2(\mu,\eta)+|v|^2+2\left\langle \int x\,d\mu-\int x\,d\nu, v\right\rangle.\]
One easily concludes from this identity that $\overline{\mu_v} = \bar{\mu}$ for any $v\in \R^d.$ According to Item (d), there exists $v \in \R^d$ such that $\eta_v:=(\mathrm{proj}_\Delta)_\#\mu_v$ has the same barycenter as $\nu$. Since $\eta_v$ is also concentrated on $\Delta$, it follows from Item (b) that $\eta_v$ belongs to $B_\nu$. Therefore, according to Item (c), 
\[
\inf_{\eta \in B_\nu}W_2^2(\mu_v,\eta) \geq \inf_{\eta(\Delta)=1}W_2^2(\mu_v,\eta) = W_2^2(\mu_v,\eta_v)
\]
and so $\eta_v = \overline{\mu_v} = \bar{\mu}.$ Finally if $X \sim \mu$, it follows from Item (a) that $(X,\mathrm{proj}_\Delta (X+v))$ is an optimal coupling between $\mu$ and $\bar{\mu}.$


The four points above can be proved as follows:
\begin{itemize}
\item[(a)] This assumption simply comes from
\[\E[|(X+v)-\bar{X}|^2]=\E[|X-\bar{X}|^2]+\underbrace{|v|^2+2\langle \E[X-\bar{X}], v\rangle}_{\text{Depends only on }\mathrm{Law}(X),\,\mathrm{Law}(\bar{X})}.\]
\item[(b)] The first implication is obvious. For the second implication, assume $\eta$ is concentrated on $\Delta$ and has the same barycenter as $\nu$. For every $x\in \Delta$, let $p_x$ be the unique probability measure concentrated on $\{y_0,\ldots, y_k\}$ with barycenter $x$ and let $\nu'=\int p_x\,d\eta(x)$. The probability measure $\nu'$ is concentrated on the same set as $\nu$. Its barycenter permits to determine it uniquely. The barycenter of $\nu'$ is $\int x\,d\eta(x)$, the barycenter of $\eta$ that is also the one of $\nu$. Therefore $\nu'=\nu$. We have proved that there exists a martingale having $\eta$ on $\nu$ as marginals. Therefore $\eta\leqc \nu$.
\item[(c)] For every $x\in \R^d$, the point $\mathrm{proj}_\Delta(x)$ is by definition the closest point in $\Delta$. So if $X' \sim \mu'$ and $Y$ takes values in $\Delta$ almost surely, one gets \[
\E[|X'-Y|^2] \geq \E[|X' - \mathrm{proj}_\Delta(X')|^2]
\]
and so $W_2^2(\mu',\eta) \geq \E[|X' - \mathrm{proj}_\Delta(X')|^2]$ for any $\eta$ concentrated in $\Delta$. In particular, $W_2^2(\mu', (\mathrm{proj}_\Delta)_\#\mu') = \E[|X' - \mathrm{proj}_\Delta(X')|^2]$ which proves the claim.

\item[(d)] Without loss of generality we can assume $k=d$ for the following reason: if $\Delta$ has positive codimension, the measure $(\mathrm{proj}_\Delta)_\#\mu'$ is exactly $(\mathrm{proj}_\Delta)_\#\circ(\mathrm{proj}_A)_\#\mu'$ where $A$ is the affine space spanned by $\Delta$. As a consequence, in our investigation we can replace $\mu$ by $(\mathrm{proj}_A)_\#\mu$, consider only translations in $A$ and see $\Delta$ as a simplex with full dimension.

Now, let $y_\nu$ be the barycenter of $\nu$ and let us prove the existence of $v \in \R^d$ such that the barycenter of  $(\mathrm{proj}_\Delta)_\#\mu_v$ is equal to $y_\nu$. This point can be proved using the notion of topological degree coming from algebraic topology (we refer to \cite[Chapter IV ]{OR09} for an introduction). 
If $\Omega$ is a bounded open set of $\R^d$ and $f:\overline{\Omega} \to \R^d$ is a continuous mapping, the degree of $f$ is a $\mathbb{Z}$-valued quantity denoted by $\dg(f,\Omega,a)$ defined for every $a \in \R^d \setminus f(\partial \Omega)$. We will use the following basic properties:
\begin{itemize}
\item if $f = \mathrm{Id}$, then $\dg(\mathrm{Id},\Omega,a) = 1$, for all $a \in \Omega$;
\item if $\dg(f,\Omega,a) \neq 0$, then the equation $f(x) = a$, $x \in \Omega$, admits at least one solution;
\item if $F: \overline{\Omega}\times [0,1] \to \R^d$ is a continuous function and $a \in \R^d$ is such that for all $t \in [0,1]$, $a \in \R^d \setminus F(t,\partial \Omega)$, then $\dg(F(\,\cdot\,,0),\Omega,a)) = \dg(F(\,\cdot\,,1),\Omega,a))$ (invariance by homotopy).
\end{itemize}
Consider $X$ a random variable with law $\mu$ and $R>0$ a positive number to be fixed later such that the open ball of center $0$ and radius $R$ denoted by $\B(0,R)$ contains $\Delta$. From now on the sphere of center $0$ and radius $R$ will be denoted by $\mathcal{S}(0,R)$. We are interested in the map
\begin{align}\label{eq:Phi}
\Phi:\overline{\B}(0,R)\times [0,1]\to \Delta\subset \R^d,
\end{align}
defined so that $\Phi(v,t)$ is the barycenter of $(\mathrm{proj}_\Delta)_\#(\mathrm{Law}(v+tX))$.
Note that this map is continuous so that it can be seen as an homotopy between $\mathrm{proj}_\Delta$ and $v\mapsto \int y\,d(\mathrm{proj}_\Delta)_\#\mu_v(y).$ Another homotopy is possible between $\Id$ and $\mathrm{proj}_\Delta$
\[\Psi:\overline{\B}(0,R)\times [0,1]\to \R^d\]
defined by $\Psi(x,t)=(1-t)x+t\mathrm{proj}_\Delta(x)$. Note that since $y_\nu$ lies in the interior of $\Delta$, it does not belong to $\Psi(\mathcal{S}(0,R)\times [0,1])$. Therefore, by invariance by homotopy, 
\[
1= \dg(\mathrm{Id},\B(0,R),y_\nu)=\dg(\Psi(\cdot,0),\B(0,R),y_\nu)=\dg(\Psi(\cdot,1),\B(0,R),y_\nu),
\]
from which we infer that $\dg(\Phi(\cdot,0),\B(0,R),y_\nu)=1$. According to Lemma \ref{lem:image} below, if $R$ is large enough then $\Phi(\mathcal{S}(0,R)\times [0,1]) \subset \partial \Delta$. Since $y_\nu$ lies in the interior of $\Delta$, we conclude that $\Phi(\mathcal{S}(0,R)\times [0,1])$ does not meet $y_\nu$, and so using the homotopy invariance again, we get that $\dg(\Phi(\cdot,1),\B(0,R),y_\nu)=1$ and so the equation $\Phi(v,1)=\int \mathrm{proj}_\Delta(x)\,d\mu_v(x) = y_\nu$ admits at least one solution $v \in \B(0,R)$ which completes the proof.\endproof
\end{itemize}

\begin{lem}\label{lem:image} Under the preceding assumptions,
if $R$ is large enough, the map $\Phi$ defined in \eqref{eq:Phi} is such that $\Phi(\mathcal{S}(0,R)\times [0,1])\subset\partial \Delta$.

\end{lem}
\begin{proof}
Recall that 
\[\Delta=\left\{\sum_{i=0}^d \lambda_i y_i:\,\sum_{i=0}^d \lambda_i=1,\,\lambda_0,\ldots,\lambda_d\geq 0\right\}.\]
For $J$ a subset of $\{0,\ldots,d\}$ we denote by $\Delta_J$ the set
\[\Delta_J=\left\{\sum_{i=0}^d \lambda_i y_i:\,\forall i\in J,\,\lambda_i=0\right\} \subset \Delta
.\]
We denote by $u_j$ the outward normal unit vector of the facet $\Delta_{\{j\}}$. Recall that the relative interior of a subset is the interior of this set in the topology induced by its affine span. Every point of $\R^d$ whose projection belongs to the relative interior of $\Delta_J$ can be written in the form $\sum_{i\notin J}\lambda_i y_i+\sum_{i\in J} \xi_i u_i$ where the coefficients $(\lambda_i)_{i\notin J}$ are positive and satisfy $\sum_{i\notin J} \lambda_i=1$ and the coefficients $(\xi_i)_{i\in J}$ are non-negative.

We prove now that for every $r>0$, there exists $R_0>0$ such that for all $R\geq R_0$ it holds \begin{itemize}
\item[-] $\Delta$ is contained in $\mathcal{B}(0,R-r+1)$,
\item[-] for every $v\in \mathcal{S}(0,R)$ there exists $j\in \{0,\ldots,d\}$ such that $\mathrm{proj}_\Delta(\mathcal{B}(v,r))\subset \Delta_{\{j\}}$.
\end{itemize}
Therefore if $X \sim \mu$ with $\mu$ such that $\mu(\mathcal{B}(0,r))=1$ and $v \in \mathcal{S}(0,R)$, there exists $j$ such that the barycenter of $(\mathrm{proj}_\Delta)_\#(\mathrm{Law}(v+tX))$ is in $\Delta_{\{j\}}\subset \partial \Delta$, which proves the claim.

Let us consider
\[g:x\in\R^d\setminus \Delta\mapsto (x-\mathrm{proj}_\Delta(x))/|x-\mathrm{proj}_\Delta(x)|\in \mathbb{S}^{d-1}.\]
For our proof it is enough to find $R$ such that for every $x\in \mathcal{S}(0,R)$ the range $g(\mathcal{B}(x,r))$ only contains vectors of $\mathbb{S}^{d-1}$ that are in the cones $(C_J)_{J\ni j,\,\#J\leq d}$ for some $j$, where 
\[C_J=\left\{u\in \R^d:\,u=\sum_{i\in J}\xi_i u_i,\,\xi_i\geq 0,\,\forall i\in J\right\}.\]
In particular it does not intersect $C_{\{0,\ldots,d\}\setminus\{j\}}$. Striving for a contradiction we assume that there exists an increasing sequence $R_n\to \infty$ and $x_n\in \mathcal{S}(0,R_n)$ such that the property is not satisfied. Therefore, for every $n$ the sets $g(\mathcal{B}(x_n,r))\subset \mathbb{S}^{d-1}$ intersects for every $j$ a cone $C_J$ with $j\in J$. However, the diameter of $g(\mathcal{B}(x_n,r))$ tends to zero. Up to selecting a subsequence, the sequence converges in the Hausdorff topology to a compact set of diameter zero, i.e a point $\{u_\infty\}\subset \mathbb{S}^{d-1}$. We have $u_\infty\in C_{J_\infty}$ for some $J_\infty\subset\{0,\ldots,d\}$. Let $j_\infty$ be in $J_\infty$. We have also $g(x_n)\to u_\infty$ (up to a subsequence) and the fact that $g$ is locally Lipschitz with a constant that tends to zero at infinity tells us that the $j_\infty$ coordinate is not zero in the cone coordinate $u=\sum_{j}\xi_j u_j$ for all the points  $u$ of $g(\mathcal{B}(x_n,r))\subset \mathbb{S}^{d-1}$ when $n$ is great enough - a contradiction.
\end{proof}

\section{Other examples and discussion of the literature}\label{sec:examples}
In this section, we briefly present and discuss other examples of optimal transport problems involving weak cost functions.
\begin{enumerate}
\item In \cite[Theorem 1.5]{GRSST15}, it is shown that if $\mu,\nu$ are probability measures on $\R$ having finite first moments and if $\bar{\mu}\in B_\nu$ denotes the projection of $\mu$ on $B_\nu$ as defined in Proposition \ref{prop:projection}, then for any even convex cost function $\theta : \R \to \R_+$, it holds
\[
\overline{\mathcal{T}}_\theta(\nu|\mu) = \inf_{\pi \in C(\mu,\bar{\mu})} \iint \theta(y-x)\,d\pi(x,y),
\]
where we recall that the barycentric optimal transport cost $\overline{\mathcal{T}}_\theta(\nu|\mu)$ is defined by \eqref{eq:barcost}. This result has been recently recovered and completed by Alfonsi, Corbetta and Jourdain in \cite{ACJ}, where an explicit expression is given for $\bar{\mu}$ (see \cite[Proposition 3.4]{ACJ}).
\item In the recent paper \cite{ABC18} (see in particular the section 5.2 of \cite{ABC18}), Alibert, Bouchitté and Champion consider the family of cost functions $(c_\lambda)_{\lambda \geq0}$ defined for all $x \in \R^d$ and $p\in \mathcal{P}(\R^d)$ having a finite first moment by
\[
c_\lambda(x,p) = (\lambda-1)\int |y-x|^2\,dp(y) + \left|\int y\,dp(y)-x\right|^2
\]
and study the associated optimal transport problem:
\[
\overline{\mathcal{T}}_{c_\lambda}(\nu |\mu) = \inf_{\pi \in C(\mu,\nu)} \int c_\lambda(x,p_x)\,d\mu(x)
\]
(this quantity is denoted by $F_\lambda(\mu,\nu)$ in \cite{ABC18}).
 It turns out that our Theorem \ref{main-result} yields a full description of optimal couplings for these costs for $\lambda>0$ (which completes the somehow implicit characterization of \cite[Theorem 5.6]{ABC18}). Namely, an easy calculation reveals that if $d\pi(x,y) = d\mu(x)dp_x(y) \in C(\mu,\nu)$ where $\mu,\nu$ are compactly supported, then  
\[
\int c_\lambda(x,p_x)\,d\mu(x) = C(\lambda)+ \int \left|\int y\,dp_{x/\lambda}(y)-x\right|^2\,d\mu_\lambda(x),
\]
where $C(\lambda) = -\lambda(\lambda-1) \int |x|^2\,d\mu(x) + (\lambda-1)\int |y|^2\,d\nu(y)$ and $\mu_\lambda$ is the image of $\mu$ under the map $x\mapsto \lambda x.$ Since the kernel $q_x = p_{x/\lambda}$, $x \in \R^d$, is such that $\nu(\,\cdot\,) = \int q_x(\,\cdot\,)\,d\mu_\lambda(x)$, we conclude that
\[
\overline{\mathcal{T}}_{c_\lambda}(\nu |\mu) = C(\lambda) + \overline{\mathcal{T}}_2(\nu |\mu_\lambda),
\]
and $p$ is optimal in $\overline{\mathcal{T}}_{c_\lambda}(\nu |\mu)$ if and only if $q$ is optimal in $\overline{\mathcal{T}}_2(\nu |\mu_\lambda)$. Moreover, Theorem \ref{main-result} yields that 
optimal kernels $p$ are all of the form
\[
p_x = r_{\nabla \varphi_\lambda (\lambda x)},\quad \text{with } r \text{ such that } \int y \,dr_u(y) = u,  \text{ for } \overline{\mu_\lambda} \text{ almost all } u \in \R^d,
\]
where $\varphi_\lambda$ is the $\mathcal{C}^1$ smooth convex function associated to the transport from $\mu_\lambda$ to the projection $\overline{\mu_\lambda}$ of $\mu_\lambda$ on $B_\nu.$
In terms of random vectors, $\pi_\lambda=\mathrm{Law}(X_\lambda,Y_\lambda) \in C(\mu,\nu)$ is optimal for $c_\lambda$ if and only if $\E[Y_\lambda | X_\lambda] = \nabla \varphi_\lambda (\lambda X_\lambda)$ (which has distribution $\overline{\mu_\lambda}$).
\item The case $\lambda =0$ is also interesting since, in this case, 
\[
\overline{\mathcal{T}}_{c_0} (\nu|\mu) = - \sup_{\pi \in C(\mu,\nu)}\int \mathrm{Var}(p_x)\,d\mu(x),
\]
where as usual $d\pi(x,y) = d\mu(x)dp_x(y)$ and for all $p \in\mathcal{P}(\R^d)$ having finite first moment, $\mathrm{Var}(p) = \int |y|^2\,dp(y) - \left| \int y\,dp(y)\right|^2$ (note that this is a concave function of $p$).
In this case, as observed in \cite{ABC18}, the unique optimal coupling $\pi^\circ$ is the product one : $\pi^\circ=\mu\otimes \nu.$

\item As proved in  \cite[Proposition 5.2]{BJ2}, the so-called \emph{shadow couplings} are solutions to a weak optimal transport problem. Shadow couplings from $\mu$ to $\nu$ are martingale transport parametrised by a measure $\hat \mu$ with marginals $\lambda$ and $\mu$. Let $(\lambda_x)_{x\in [0,1]}$ be a disintegration of $\hat \mu$ with respect to $\mu$. Then
\[
c^{\hat\mu}(x,p_x)=\inf \int (1-u)\sqrt{1+y^2}\,d\alpha(u,y)
\]
where the $\inf$ goes among all $\alpha$ with first marginal $\lambda_x$ and second marginal $p_x$ and such that we have $\int f(u)(y-x)\,d\alpha(u,y)=0$ for every bounded $f$. 
\item Last but not least, let us mention that after the completion of this work, we learned that a recent remarkable economics paper on optimal mechanisms for the multiple-good monopoly problem shows interesting similarities with our context. In \cite{DDT17}, Daskalakis, Deckelbaum and Tzamos study the maximization problem among convex coordinate-wise nondecreasing $1$-Lipschitz potential functions defined on some $d$-dimensional rectangle. In their Theorem 2, they prove the strong duality between this problem and ``strong'' dual problem (for us a primal problem). This problem is a transport problem with $\ell_1$-norm as cost function and the possibility to replace $\mu$ and $\nu$ by $\mu'$ and $\nu'$ with $\mu \leqcs \mu'$ and $\nu'\leqcs \nu$.

\end{enumerate}

\section{Proofs}\label{sec:proofs}
This section contains the proofs of Proposition \ref{prop:projection} and Theorem \ref{main-result} and of the technical Lemma \ref{lem:contraction}. 

\subsection{Proof of the existence and uniqueness of the projection}
In what follows, we will equip $\mathcal{P}_2(\R^d)$ with the topology generated by the metric $W_2$. We recall (see \textit{e.g.} \cite[Theorem 7.12]{Vil03}) that if $(\eta_k)_{k\geq 1}$ is a sequence of elements of $\mathcal{P}_2(\R^d)$ and $\eta \in \mathcal{P}_2(\R^d)$, then $W_2(\eta_k,\eta) \to 0$ if and only if $\int f\,d\eta_k \to \int f\,d\eta$ for any continuous function $f$ satisfying $|f|(x) \leq a+ b|x|^2$, $x \in \R^d$, for some $a,b\geq0$. 

\proof[Proof of Proposition \ref{prop:projection}]
Let us show that the set $B_\nu$ is compact in $\mathcal{P}_2(\R^d)$ for the topology induced by $W_2$. 

First, let us show that $B_\nu$ is closed. Namely, $B_\nu$ can be written as follows
\[
B_\nu =\left \{\eta \in \mathcal{P}_2(\R^d) : \int f\,d\eta \leq \int f\,d\nu, \forall f: \R^d \to \R \text{ convex and Lipschitz}\right\}.
\]
(The fact that one can restrict to convex and Lipschitz functions in the definition of the convex order is classical ; one can see this by noting that if $f$ is convex then $f_n:x\mapsto \inf_{y\in \R^d} \{f(y) + n|x-y|\}$, $x\in \R^d$, is a sequence of $n$-Lipschitz convex functions converging to $f$ monotonically.) Since functionals $\eta \mapsto \int f\,d\eta$ with $f$ convex and Lipschitz are continuous for the $W_2$ topology, it follows that $B_\nu$ is closed.

Now let us show that $B_\nu$ is precompact. According to a variant of Prokhorov theorem (see \textit{e.g.} \cite[Theorem 9.10]{GRST17}), this is equivalent to show that 
\[
\sup_{\eta \in B_\nu} \int |x|^2 \mathbf{1}_{|x|>R}\,d\eta(x)  \to 0
\]
as $R$ goes to $\infty.$ Note that
\[
 |x|^2 \mathbf{1}_{|x|>R} \leq \max(0 ; 2|x|-R)^2,\qquad \forall x \in \R^d.
\]
The function on the right hand side being convex, one thus gets
\[
\sup_{\eta \in B_\nu} \int |x|^2 \mathbf{1}_{|x|>R}\,d\eta(x) \leq \int  \max(0 ; 2|x|-R)^2\,d\nu(x) := \varepsilon(R).
\]
The monotone convergence theorem then implies that $\varepsilon(R) \to 0$ as $R \to +\infty$, which shows that $B_\nu$ is precompact.

It follows from what precedes that $B_\nu$ is compact. The map $\eta \mapsto W_2(\eta,\mu)$ therefore reaches its minimum on $B_\nu$.

Now let us prove that the minimizer is unique. This will follow from the strict convexity of $W_2^2(\mu,\,\cdot\,)$ along generalized geodesics with base point $\mu$ and from the convexity of the set $B_\nu$ along those generalized geodesics. More precisely, suppose that $\eta_0,\eta_1$ are in $B_\nu$ and let $(X,Y_0,Y_1)$ be a random vector such that $X\sim \mu$, $Y_0 \sim \eta_0$, $Y_1\sim \eta_1$ and so that $\mathrm{Law} (X,Y_i) \in C(\mu,\eta_i)$  and $W_2^2(\mu,\eta_i) = \E[|X-Y_i|^2]$, $i=0,1.$ Then define $\eta_t = \mathrm{Law} ((1-t) Y_0 + tY_1)$ for all $t \in [0,1]$. First note that $\eta_t \in B_\nu$ for all $t\in [0,1]$. Indeed, if $f$ is a convex function on $\R^d$, it holds
\[
\int f \,d\eta_t = \E[f((1-t) Y_0 + tY_1)] \leq (1-t) \E[f(Y_0)] + t \E[f(Y_1)]  \leq  \int f\,d\nu.
\]
Moreover, according to Lemma 9.2.1 of \cite{AGS08}, it holds
\[
W_2^2 (\mu, \eta_t) \leq (1-t)W_2^2(\mu,\eta_0) + t W_2^2(\mu,\eta_1) -t(1-t) W_2^2(\eta_0,\eta_1)
,\qquad \forall t\in [0,1]
\]
So if $W_2^2(\mu,\eta_0) = W_2^2(\mu,\eta_1) = \inf_{\eta \in B_\nu} W_2^2(\mu,\eta)$, then necessarily $\eta_0=\eta_1$, which shows uniqueness. 

Finally, let us show that $\overline{\mathcal{T}}_2(\nu|\mu) = W_2^2(\bar{\mu},\mu).$ 
First, let $(X,Y)$ be a coupling of $\mu$ and $\nu$ and set $X' = \E[Y|X]$. Since $(X',Y)$ is a martingale, it follows that $\mathrm{Law}(X') \in B_\nu$.  Therefore $\E[|X - X'|^2]\geq W_2^2(\mu,\bar{\mu})$ and so optimizing gives $\overline{\mathcal{T}_2}(\nu|\mu) \geq W_2^2(\mu,\bar{\mu}).$  On the other hand, let $\pi \in C(\mu,\eta)$ be a coupling between $\mu$ and some $\eta \in B_\nu$. Since $\eta \leqc \nu$, one can construct a Markov chain $(X,X',Y)$ such that $(X,X')\sim \pi$ and $(X',Y)$ is a martingale. Then it holds,
\begin{align*}
\E[|X-X'|^2] &= \E[|X - \E[Y|X,X']|^2] \\
& = \E\left[\E[[|X - \E[Y|X,X']|^2 | X]\right]\\
& \geq \E[|X - \E[Y|X]|^2] \geq \overline{\mathcal{T}}_2(\nu|\mu),
\end{align*}
where the inequality comes from Jensen inequality for conditional expectation. Optimizing over $X'$ and $\eta$ gives that $\overline{\mathcal{T}_2}(\nu|\mu) \leq W_2^2(\bar{\mu},\mu)$ and completes the proof. \endproof
\begin{rem}
Note that it is easy to conclude using similar arguments that there always exists a deterministic map $T$ transporting $\mu$ on $\bar{\mu}$ such that $W_2^2(\mu,\bar{\mu}) = \int |x-T(x)|^2\,d\mu(x)$. Indeed, suppose that $(X,\bar{X})$ is an optimal coupling for $W_2^2(\bar{\mu},\mu)$ and consider $X' = \E[\bar{X}|X]$. 
Then, the conditional Jensen inequality gives
\begin{equation}\label{rem:transport}
\E[|X - X'|^2] = \E[|X - \E[\bar{X}|X]|^2] \leq \E[|X-\bar{X}|^2] = W_2^2(\bar{\mu},\mu).
\end{equation}
Let $\eta' := \mathrm{Law}(X')$. Then $\eta' \leqc \bar{\mu}$ and $\bar{\mu} \leqc \nu$ so $\eta' \in B_\nu$. Therefore, $\E[|X-X'|^2] \geq \inf_{\eta \in B_\nu} W_2^2(\eta,\mu) = W_2^2(\bar{\mu},\mu)$ and there is equality in \eqref{rem:transport}. So $X' = \E[\bar{X}|X] \sim \bar{\mu}$ and $(X,X')$ is an optimal coupling. Finally, by definition of conditional expectation, there exists a measurable $T:\R^d\to \R^d$ such that $\E[\bar{X}|X] = T(X)$ almost surely, which proves the claim.
\end{rem}

\subsection{Proof of the main result}

Our proof of Theorem \ref{main-result} follows closely the scheme developed by Gangbo \cite{Gangbo94} in his alternative proof of Brenier Theorem: first we show the dual attainment, and then we obtain the existence of the transport map $\nabla \varphi$ by doing a first variation around the minimizer $f^\circ$.

As explained in Remark \ref{rem:deux_versions}, we will treat the case of probability measures with compact supports and the general case separately. Only Items (a) and (b) of Theorem \ref{main-result} are concerned. The proof of Item (c) being identical in the compact and the general case it is proved only once on page \pageref{proof:c}.

\subsubsection{The compact case}In this section, we assume that $\mu$ and $\nu$ have compact supports. In the following theorem we restate Theorem \ref{main-result} (a) and (b) in this special case. In comparison with Theorem \ref{main-result} (a), note that in the compact case $f^\circ$  is always finite valued and bounded from below.
\begin{thm}\label{thm:compact}
Let $\mu$ and $\nu$ be probability measures with compact support on $\R^d$.
\begin{itemize}
\item[(a)] There exists some convex function $f^\circ:\R^d\to \R$ bounded from below such that
\begin{equation}\label{eq:attainment}
\overline{\mathcal{T}_2}(\nu|\mu) = \int Q_2f^\circ\,d\mu - \int f^\circ\,d\nu,
\end{equation}
where, for any function $g:\R^d\to \R$,
\[
Q_2g(x) = \inf_{y\in \R^d}\{ g(y) + |y-x|^2\},\qquad  x \in \R^d.
\]
If the supports of $\mu$ and $\nu$ are contained in the closed ball of radius $R>0$ centered at $0$ then $f^\circ$ is $4R$-Lipschitz.
\item[(b)] Let $h$ and $\varphi$ be the convex functions defined by 
\[
h(x) = \frac{f^\circ(x)+|x|^2}{2}, \quad x\in \R^d \qquad\text{and}\qquad \varphi(y) = h^*(y),\quad y\in \R^d.
\]
The function $\varphi$ is $\mathcal{C}^1$-smooth on $\R^d$ and the map $\nabla \varphi$ is $1$-Lipschitz on $\R^d$. The projection $\bar{\mu}$ of $\mu$ on $B_\nu$ is such that 
\[
\bar{\mu} = \nabla \varphi_\# \mu.
\]
\end{itemize}
\end{thm}
\begin{proof}[Proof of Item (a) of Theorem \ref{thm:compact}]
We will assume that $\mu$ and $\nu$ are supported in a closed ball $B$ of radius $R>0$ centered at $0.$ Let $\alpha$ be the convex function on $[0,\infty)$ defined by
\[
\alpha(t) = t^2 \text{ if } t\in [0,2R]\qquad \text{ and }\qquad \alpha(t) = 4Rt-4R^2 \text{ if } t\geq 2R. 
\] 
and let $\theta(u)=\alpha(|u|)$, for all $u\in \R^d.$ Since $\mu$ and $\nu$ are supported in $B$ and we have the duality statement Theorem \ref{them:duality} it is easily seen that 

\begin{align}\label{eq:dual}
\overline{\mathcal{T}}_2(\nu|\mu)=\overline{\mathcal {T}}_{\theta}(\nu|\mu) = \sup_{g\in \mathcal{G}} \left\{\int Q_{\theta}g\,d\mu - \int g\,d\nu\right\},
\end{align}
where the supremum runs over the set $\mathcal{G}$ of convex functions $g$ bounded from below.

\noindent \textit{Step 1 - Preparation.} 
Let $P_{\theta}$ be the operator acting on functions defined as follows
\[
P_{\theta}g(y) = \sup_{x\in \R^d}\{g(x) - \theta(y-x)\},\qquad y\in \R^d.
\]
First, let us show that we can refine the duality formula \eqref{eq:dual} by modifying the class of functions $\mathcal{G}$ as follows:
\[
\overline{\mathcal {T}}_2(\nu|\mu) = \sup_{f \in \mathcal{F}} \left\{ \int Q_{\theta}f\,d\mu - \int f \,d\nu\right\},
\]
where $\mathcal{F}$ is the subset of $\mathcal{G}$ of functions $f$ which are convex, bounded from below, and satisfy both $f(0)=0$ and $f=P_{\theta}(Q_{\theta}f)$. Indeed, let $g$ be an element of $\mathcal{G}$. Then, by definition of $Q_{\theta}$, it holds
\[
Q_{\theta}g(x)-\theta(y-x) \leq g(y),\qquad \forall x,y\in \R^d
\]
and so $g\geq P_{\theta}(Q_{\theta}(g))$. Since $Q_{\theta}g$ is convex (as an infimum convolution of convex functions), the function $f$ defined by
\[
f(y) = P_{\theta}(Q_{\theta}g)(y) = \sup_{u\in \R^d}\{Q_{\theta}g(y-u) - \theta(u)\},\qquad y\in \R^d
\]
is also convex as a supremum of convex functions and bounded from below because $P_\theta$ and $Q_\theta$ preserve this property. Moreover, it is easily seen that $Q_{\theta}f=Q_{\theta}g$ and so $P_{\theta}(Q_{\theta}f)=f$. Therefore,
\[
\int Q_{\theta}g\,d\mu - \int g\,d\nu \leq \int Q_{\theta}f\,d\mu - \int f\,d\nu=\int Q_{\theta}(f-f(0))\,d\mu - \int (f-f(0))\,d\nu,
\]
which shows that the duality formula can be restricted to functions $f\in \mathcal{F}$.

\medskip

\noindent \textit{Step 2 - Dual attainment.} 
Now, let us show that there is some convex function $f^\circ$ satisfying \eqref{eq:attainment}.
First of all, if $f \in \mathcal{F}$, then $f$ is $4R$-Lipschitz. This comes from the fact that
\[
f(y) = \sup_{x\in \R^d} \{Q_{\theta}f(x) - \theta(y-x)\},\qquad \forall y\in \R^d
\]
which (the function $y\mapsto \theta(y-x)$ being $4R$-Lipschitz on $\R^d$ for every $x\in \R^d$) shows that $f$ is a supremum of $4R$-Lipschitz functions and is thus $4R$-Lipschitz itself. Since $f(0)=0$, this implies in particular that $|f| \leq 20R^2$ on $B'=5B$ (the ball of radius $5R$ centered at $0$).
Also, since $f$ is $4R$-Lipschitz, it holds
\[
f(y)+|y-x|^2\geq f(x) - 4R|y-x|+|y-x|^2 \qquad \forall x,y\in \R^d.
\]
Therefore, $f(y) + |y-x|^2 \geq f(x)$ whenever $|y-x|\geq 4R.$ Since $Q_2f(x)\leq f(x)$, this implies that 
\[
Q_2f(x) = \inf_{|y-x|\leq 4R} \{f(y) + |y-x|^2\},\qquad \forall x\in \R^d.
\]
Now, let $f_n \in \mathcal{F}$ be some minimizing sequence. The functions $f_n$ are $4R$-Lipschitz and uniformly bounded on the ball $B'$. Therefore, it follows from Ascoli's theorem that $f_n$ admits a sub-sequence (still denoted $f_n$ in what follows) converging uniformly to some $f^\circ$ on $B'.$ This function $f^\circ$ is convex on $B'$ as a pointwise limit of convex functions. It is easily seen that $Q_2f_n \to Q_2f^\circ$ (uniformly) on $B$. Since $Q_2f_n\geq Q_{\theta}f_n$ pointwise, it holds
\[
\overline{\mathcal{T}_2}(\nu|\mu) \geq \int Q_2f_n\,d\mu - \int f_n\,d\nu \geq \int Q_{\theta}f_n\,d\mu - \int f_n\,d\nu.
\]
Since the right hand side goes to $\overline{\mathcal{T}_2}(\nu|\mu)$ and $\mu,\nu$ are supported in $B$, letting $n\to \infty$ yields to \eqref{eq:attainment}. For the moment the convex function $f^\circ$ is defined only on $B'$ but it can be easily extended outside $B'$ as follows: the function $\tilde{f}(x) = \inf_{y \in B'}\{f^\circ(y) + 4R|x-y|\}$, $x \in \R^d$, is convex as an infimum convolution of two convex functions, and since $f^\circ$ is $4R$-Lipschitz it is easily seen that $\tilde{f} = f^\circ$ on $B'$. We can thus assume that $f^\circ$ is a finite valued  convex and bounded from below function defined on the whole $\R^d$. This completes the proof.
\end{proof}

\medskip
In order to prove Item (b) of Theorem \ref{thm:compact}, we will need a technical lemma  adapted from \cite[Lemma 2.4]{Gangbo94}.
\begin{lem}\label{Gangbo}
Let $f:\R^d\to \R\cup\{+\infty\}$ be a lower semi-continuous convex function and let $h(x)=\frac{f(x)+|x|^2}{2}$, $x\in \R^d$.
\begin{itemize}
\item[(a)] The function $h^*$ is $\mathcal{C}^1$ smooth and it holds 
\begin{equation}\label{eq:Gangbo}
Q_2f(x) = f(\nabla h^*(x))+|\nabla h^*(x)-x|^2,\qquad \forall x \in \R^d.
\end{equation}
\item[(b)] For all continuous function $u:\R^d\to \R$ such that $u+f$ is convex and such that $u(x)\geq -a |x| - b$, $x\in \R^d$, for some $a,b \geq0$ it holds
\begin{equation}\label{eq:Gangbolim}
\lim_{t\to 0^+}\frac{Q_2(f+tu)(x) - Q_2(f)(x)}{t} =u(\nabla h^*)(x),\qquad \forall x\in \R^d.
\end{equation}
Furthermore, there exist some $\alpha,\beta\geq0$ such that  
\[
\frac{Q_2(f+tu)(x) - Q_2(f)(x)}{t} \geq - \alpha |x| -\beta,\qquad \forall x \in \R^d,\qquad \forall t \in ]0,1].
\]
These conclusions also hold for $u_k = -\min(f,k)$, $k\in \N$ in place of $u$ (even though $u_k$ is not necessarily continuous).
\end{itemize}
\end{lem}

Let us admit the lemma until page \pageref{proof_of_lemma} where the proof is postponed and first prove Item (b) of Theorem \ref{thm:compact}.

\begin{proof}[Proof of Item (b) of Theorem \ref{thm:compact}] According to Item (a) of Lemma \ref{Gangbo}, the function $\varphi=h^*$ is of class $\mathcal{C}^1$. The fact that $\nabla \varphi$ is $1$-Lipschitz follows from Lemma \ref{lem:contraction}. Set $\tilde{\mu} = \nabla \varphi_\# \mu$ and let us show that $\tilde{\mu} = \bar{\mu}$. First let us prove that $\tilde{\mu} \preceq_c \nu$.
Let $u:\R^d\to \R$ be some arbitrary convex function. 
By optimality of $f^\circ$ it holds, for all $t>0$,
\[
\int Q_2(f^\circ+tu)\,d\mu - \int (f^\circ+tu)\,d\nu \leq \int Q_2(f^\circ)\,d\mu - \int f^\circ\,d\nu
\]
Therefore, for all $t>0$, 
\[
\int \frac{Q_2(f^\circ+tu) - Q_2(f^\circ)}{t}\,d\mu \leq \int u\,d\nu.
\]
Since $u$ is bounded from below by some affine function it satisfies the assumption of Item (b) of Lemma \ref{Gangbo} and one concludes using Fatou's Lemma that
\[
\int u(\nabla \varphi)\,d\mu \leq \int u\,d\nu
\]
for all convex function $u:\R^d\to \R$. This shows $\tilde{\mu} \leqc \nu$. In particular, one gets $\int f^\circ\,d\tilde{\mu} \leq \int f^\circ\,d\nu$. Actually, for this special function, equality holds. Indeed, the function $u = -f^\circ$ is Lipschitz according to Item (a) of Theorem \ref{thm:compact} and such that $f^\circ+t u$ is convex for all $0\leq t\leq 1$. So
it satisfies the assumptions of Item (b) of Lemma \ref{Gangbo}. Reasoning as above gives us $\int f^\circ\,d\tilde{\mu} \geq \int f^\circ\,d\nu,$ which shows equality. Therefore,
\begin{align*}
\overline{\mathcal{T}}_2(\nu|\mu) & = \int Q_2f^\circ\,d\mu - \int f^\circ\,d\nu \\
& = \int f^\circ(\nabla \varphi(x))+|\nabla \varphi(x)-x|^2\,d\mu(x) - \int f^\circ(y)\,d\nu(y)\\
& = \int |\nabla \varphi(x)-x|^2\,d\mu(x)\\
& \geq W_2^2(\tilde{\mu},\mu).
\end{align*}
Finally, if $\eta \leqc \nu$, then 
\begin{align*}
\overline{\mathcal{T}}_2(\nu|\mu) &= \sup_{g \text{ convex }} \left\{ \int Q_2 g\,d\mu - \int g\,d\nu\right\}\\
& \leq \sup_{g \text{ convex }} \left\{ \int Q_2 g\,d\mu - \int g\,d\eta\right\}\\
& \leq W_2^2(\mu,\eta),
\end{align*}
where the last inequality follows easily from the inequality $Q_2g(x) - g(y) \leq |y-x|^2$, $x,y \in \R^d.$
In particular taking $\eta = \tilde{\mu}$ shows that 
\[
\overline{\mathcal{T}}_2(\nu|\mu) = W_2^2(\tilde{\mu},\mu) = \inf_{\eta \leqc \nu} W_2^2(\eta,\mu)
\]
and completes the proof.
\end{proof}

Now let us prove Lemma \ref{Gangbo}.
\proof[Proof of Lemma \ref{Gangbo}]\label{proof_of_lemma}
Let $u:\R^d \to \R$ be a function such that $u+f$ is convex and lower semi-continuous and $u(x) \geq -a |x| - b$, $x \in \R^d$, for some $a,b\geq0.$  Note that for the moment $u$ is not necessarily continuous.

\noindent (a) For all $t \in [0,1]$, let $h_t$ be the lower semi-continuous function defined by 
\[
h_t(x) =  \frac{1}{2} \left((1-t) f(x) + t (u(x)+f(x))+ |x|^2\right),\qquad x \in \R^d,
\]
with thus $h_0=h.$
 Applying  Lemma \ref{lem:contraction} to $g = h_t$ gives that $h_t^*$ is of class $\mathcal{C}^1$ on $\R^d$ and $\nabla h_t^*$ is $1$-Lipschitz. 
Moreover, 
\begin{align*}
Q_2(f+tu)(x) & = \inf_{y\in \R^d}\{f(y) +tu(y) + |x-y|^2\}\\
& =|x|^2 - 2\sup_{y\in \R^d}\{x\cdot y -  h_t(y) \} =  |x|^2-2h_t^*(x)
\end{align*}
and one sees that $y$ is optimal  if and only if $x\cdot y = h_t(y) + h_t^*(x)$, that is to say if and only if $y \in \partial h_t^*(x) = \{\nabla h_t^*(x)\}$ (see \textit{e.g} \cite[Corollary E 1.4.4]{HUL01}). So it holds, for all $t\in[0,1]$
\begin{equation}\label{eq:Gangbo2}
Q_2(f+tu)(x) = f(\nabla h_t^*(x)) + tu (\nabla h_t^*(x))+ |x-\nabla h_t^*(x)|^2,\qquad \forall x \in \R^d,
\end{equation}
which gives in particular \eqref{eq:Gangbo} for $t=0.$
\medskip

\noindent (b) Observe that, according to \eqref{eq:Gangbo}, it holds
\begin{align*}
Q_2(f+tu)(x) &= \inf_{y\in \R^d} \{(f+tu)(y) + |x-y|^2\}\\
 &\leq (f+tu)(\nabla h^*(x)) + |x- \nabla h^*(x)|^2\\
& = Q_2f(x) + tu(\nabla h^*(x))
\end{align*}
and so
\begin{equation}\label{eq:limsup}
 \frac{Q_2(f+tu)(x) - Q_2f(x)}{t} \leq u(\nabla h^*(x)),\qquad \forall x\in \R^d, \forall t\in (0,1].
\end{equation}
Similarly, using \eqref{eq:Gangbo2}, one easily gets
\begin{equation}\label{eq:Binf}
 \frac{Q_2(f+tu)(x) - Q_2f(x)}{t} \geq u(\nabla h_t^*(x)),\qquad \forall x\in \R^d, \forall t\in (0,1].
\end{equation}
Let us show that for any fixed $x \in \R^d$,  $\nabla h_t^*(x) \to \nabla h^*(x)$ as $t\to 0^+$.
The function $f$ being convex it is bounded from below by some affine function. Using the lower bound on $u$, it is easily seen that there exist some constants $c,c' \geq 0$, such that 
\[
f(y) + tu(y) \geq -c |y| - c',\qquad \forall y\in \R^d.
\]
Therefore, since $Q_2 (f+tu)(x) \leq f(x)+tu(x) \leq |f|(x) + |u|(x)$, Identity \eqref{eq:Gangbo2} gives that
\[
 |f|(x) + |u|(x) \geq -c |\nabla h_t^*|(x) - c' +|x- \nabla h_t^*(x)|^2.
\]
Therefore, $\sup_{t\in [0,1]} |\nabla h_t^*|(x) <+\infty.$ It follows that if $t_n \in (0,1]$ is some sequence converging to $0$, the sequence $\nabla h_{t_n}^*(x)$ admits a subsequence (still denoted $\nabla h_{t_n}^*(x)$ for simplicity) converging to some $\ell(x)$. According to \eqref{eq:limsup} and \eqref{eq:Binf} and the lower bound on $u$, $Q_2(f+t_nu)(x) \to Q_2f(x)$. On the other hand, 
\begin{equation}\label{eq:Gangbo3}
Q_2(f+t_n u)(x) = f(\nabla h_{t_n}^*(x)) + |x-\nabla h_{t_n}^*(x)|^2 + t_n u(\nabla h_{t_n}^*(x))
\end{equation}
and so taking the $\liminf$, it is easily seen (using the lower bound on $u$) that
\[
\liminf_{n\to +\infty}Q_2(f+t_n u)(x) \geq  \liminf_{n\to +\infty}f(\nabla h_{t_n}^*(x)) + |x-\ell(x)|^2 \geq f(\ell(x)) + |x-\ell(x)|^2.
\]  
So we get $Q_2f(x) \geq  f(\ell(x)) + |x-\ell(x)|^2$ and thus $Q_2f(x) =  f(\ell(x)) + |x-\ell(x)|^2$ which implies according to the proof of Item (a) that $\ell(x) = \nabla h^*(x).$ It follows that $\nabla h^*(x)$ is the unique limit point of the family $(\nabla h_{t}^*(x))_{t \in (0,1]}$ when $t \to 0^+$, and so $\nabla h_{t}^*(x) \to \nabla h^*(x)$ as $t \to 0^+.$ 
To prove \eqref{eq:Gangbolim}, let us discuss the different cases.
\begin{itemize}
\item If $u$ is continuous, then taking the limit in \eqref{eq:limsup} and \eqref{eq:Binf} gives \eqref{eq:Gangbolim}.
\item If $u = -\min(f;k)$ for some $k\in \N$, then $u+f = \max(0 ; f-k)$ is clearly convex and lower semi-continuous and $u \geq - k$. Also, since $f$ is lower bounded by some affine function, the function $u$ is bounded from above by a function of the form $x\mapsto  e|x| + e'$, for some $e,e'\geq0.$  Therefore, if $t_n \to 0$, then $ t_n u(\nabla h_{t_n}^*(x))
 \to 0$ and so, taking the limit in \eqref{eq:Gangbo3}, one easily gets that sees that $f(\nabla h_{t_n}^*(x)) \to f(\nabla h^*(x))$ and so, by definition of $u$, $u(\nabla h_{t_n}^*(x)) \to u(\nabla h^*(x))$ as $n\to +\infty.$ Taking the limit in \eqref{eq:limsup} and \eqref{eq:Binf} gives \eqref{eq:Gangbolim} as above.
\end{itemize}
Finally, in both cases, since $x\mapsto \nabla h_t^*(x)$ is $1$-Lipschitz, \eqref{eq:Binf} gives that
\[
 \frac{Q_2(f+tu)(x) - Q_2f(x)}{t} \geq u(\nabla h_t^*(x)) \geq - a | \nabla h_t^*(x) | - b \geq - a |x| - a|\nabla h_t^*(0)| -b
\]
and since $\sup_{t\in [0,1]} |\nabla h_t^*|(0) <+\infty$ this completes the proof.
\endproof

\subsubsection{Proof of Theorem \ref{main-result} in the general case}
We follow the same path as before and first show the existence of a dual optimizer. As explained in Remark \ref{rem:deux_versions}, another proof has been proposed in \cite{BBP}. It relies on the fact that Theorem \ref{main-result} on measures with finite second moments can be built by starting from Theorem \ref{thm:compact} on compact measures.
\begin{proof}[Proof of Item (a) of Theorem \ref{main-result}]
First let us justify that $\int Q_2f \,d\mu$ is finite for any lower semi-continuous convex function $f:\R^d \to \R\cup \{+\infty\}$ such that $f$ is integrable with respect to $\nu.$
Indeed, selecting some $y_o$ such that $f(y_o)$ is finite one first get that $Q_2f(x) \leq  f(y_o) + |y_o-x|^2$, $x\in \R$. On the other hand, the function $Q_2 f$ is convex on $\R^d$ as an infimum convolution of two convex functions, therefore it is bounded from below by some affine function.  Since $\mu \in \mathcal{P}_2(\R^d)$, this shows that $Q_2 f$ is integrable with respect to $\mu.$ 

Without loss of generality we assume that $\int x\, d\nu(x)=0$. According to \cite[Theorem 2.11 (3)]{GRST17}, 
\[
\overline{\mathcal{T}}_2(\nu |\mu) = \sup_{f} \left\{ \int Q_2  f\,d\mu - \int  f\,d\nu \right\},
\]
where the infimum runs over the set of convex functions which are Lipschitz and bounded from below. Without loss of generality one can also impose that $f(0)=0.$ Let $(f_n)_{n\in \N}$ be a maximizing sequence and let $g_n = Q_2f_n$. Note that, since $f_n$ is convex, $g_n$ is also convex as an infimum convolution of two convex functions.

Let us show that for all $x \in \R^d$, the sequence $g_n(x)$ is bounded.
First of all, observe that, for all $n$, it holds (choosing $y=0$)
\begin{equation}\label{eq:existence-1}
g_n(x)=Q_2f_n(x) =\inf_{y\in \R^d} \{ f_n(y) + |y-x|^2\} \leq |x|^2,\qquad \forall x \in \R^d,
\end{equation}
which shows that for all $x \in \R^d$, the sequence $g_n(x)$ is bounded from above.
On the other hand, since $|x|^2 - g_n(x) = \sup_{y\in \R^d} \{2x\cdot y -(f_n(y)+|y|^2)\}$, it follows that the function $|x|^2 - g_n(x)$ is convex as a supremum of convex functions. Therefore, denoting by $b_\mu = \int x\,d\mu$, Jensen inequality yields
\[
\int |x|^2 - g_n(x) \,d\mu(x) \geq |b_\mu|^2 - g_n(b_\mu)
\] 
and so
\[
g_n \left(b_\mu\right) \geq  |b_\mu|^2 -\int |x|^2 \,d\mu(x) + \int g_n(x)\,d\mu(x).
\]
Now, observe that $\int g_n\,d\mu \geq  \int g_n\,d\mu - \int f_n\,d\nu$, since according to Jensen inequality it holds
\[
\int f_n\,d\nu \geq f_n\left(\int x\,d\nu \right) = f_n(0)=0.
\]
Since the sequence $\int g_n\,d\mu - \int f_n\,d\nu$ converges to $\overline{\mathcal{T}}_2(\nu |\mu)$ it is bounded from below. This implies that the sequence $g_n(b_\mu)$ is also bounded from below. Now let $x$ be an arbitrary point in $\R^d$ and write $ b_\mu = \frac{1}{2} x + \frac{1}{2} (2b_\mu-x)$, then the convexity of $g_n$ and \eqref{eq:existence-1} yield
\[
g_n(b_\mu) \leq \frac{1}{2}g_n(x) +\frac{1}{2}g_n(2b_\mu-x) \leq \frac{1}{2}g_n(x) +\frac{1}{2}|2b_\mu-x|^2 
\]
and so the sequence $g_n(x)$ is bounded from below.

The sequence of convex functions $g_n$ thus satisfies the following boundedness properties:
\[
-\infty <\inf_{n \in \N} g_n(x)\qquad \text{and}\qquad \sup_{n \in \N} g_n(x) <+ \infty,\qquad \forall x \in \R^d.
\]
According to \cite[Theorem 10.9 (p.90)]{Rockafellar70}, it is possible to extract a subsequence from $g_n$ (that we will still denote by $g_n$) which converges pointwise to a convex function $g:\R^d \to \R$ (the convergence is also uniform on every compact set, but this will not be needed in the sequel).

According to \eqref{eq:existence-1}, $g_n(x) \leq |x|^2$ and $\int |x|^2\,d\mu <+\infty$, so on the one hand Fatou's lemma yields
\[
\limsup_{n\to +\infty} \int g_n\,d\mu \leq \int \limsup_{n\to +\infty}g_n\,d\mu = \int g\,d\mu <+\infty.
\]
On the other hand,
\begin{equation}\label{eq:existence-2}
f_n(x) \geq g_n(y) -|y-x|^2,\qquad \forall x,y\in \R^d.
\end{equation}
Since for a fixed $y_o$, the sequence $g_n(y_o)$ is bounded from below and $\int |y_o-x|^2\,d\nu(x)<+\infty$, Fatou's lemma provides
\[
\liminf_{n\to +\infty} \int f_n\,d\nu \geq \int \liminf_{n\to +\infty} f_n \,d\nu .
\]
Moreover, taking the $\liminf$ in \eqref{eq:existence-2} and then optimizing over $y \in \R^d$ gives
\[
\liminf_{n\to +\infty} f_n(x) \geq P_2g(x)= \sup_{y\in \R^d}\{g(y) - |y-x|^2\},\qquad \forall x\in \R^d.
\]
Finally,
\begin{align*}
\overline{\mathcal{T}}_2(\nu|\mu) & = \limsup_{n\to +\infty} \left(\int g_n \,d\mu - \int f_n \,d\mu \right) \\
& \leq \limsup_{n\to +\infty} \int g_n \,d\mu - \liminf_{n\to +\infty}\int f_n \,d\mu \\
& \leq \int g\,d\mu - \int P_2g\,d\nu.
\end{align*}
Let $f^\circ$ be the function $x\mapsto P_2g(x) =  \sup_{u\in \R^d}\{g(x+u) - |u|^2\} \in \R\cup\{+\infty\}$, $x \in \R^d$. This function is convex and lower semi-continuous as a supremum of convex and continuous functions. Since $g(x) \leq f^\circ(y)+ |y-x|^2$ for all $y \in \R^d$, optimizing over $y\in \R^d$ yields $g \leq Q_2 f^\circ$. 
Therefore,
\[
\int Q_2 f^\circ\,d\mu - \int f^\circ \,d\nu \geq \overline{\mathcal{T}}_2(\nu|\mu),
\]
which completes the proof.
\end{proof}

\begin{proof}[Proof of Item (b) of Theorem \ref{main-result}] The beginning of the proof of Item (b) of Theorem \ref{thm:compact} can be repeated exactly as before and yields to the conclusion that $\bar \mu := \nabla \varphi_\# \mu \leqc \nu.$ This shows in particular that $\int f^\circ\,d\bar \mu \leq \int f^\circ \,d\nu.$ To show that this is actually an equality, it is no longer possible to take $u=- f^\circ$, since the function $f^\circ$ is not necessarily Lipschitz. Instead, let us take $u_k = - \min( f^\circ ; k)$. The function $u_k$ is such that $f^\circ+tu_k$ is convex for all $t \in [0,1]$, so it is admissible to perform first variation in the optimization problem. Applying Lemma \ref{Gangbo} and reasoning as in the proof of Theorem \ref{thm:compact}, one sees that $\int f^\circ \,d\bar \mu \geq  \int \min( f^\circ,k) \,d\bar \mu \geq \int \min(f^\circ,k) \,d\nu$. Letting $k$ go to $+\infty$ gives the desired equality. The rest of the proof remains unchanged.
\end{proof}

\begin{proof}[Proof of Item (c) of Theorem \ref{main-result}] \label{proof:c}
Since the probability $\bar{\mu}$ given by Proposition \ref{prop:projection} is dominated by $\nu$ for the convex order, Strassen Theorem implies that there exists a transport kernel $q$ such that $\int y\,dq_x(y)=x$ for $\bar{\mu}$ almost every $x$ and $\nu(\,\cdot\,) = \int q_x(\,\cdot\,)\,d\bar{\mu}(x)$. Let $(X,\bar{X},Y)$ be a time inhomogeneous Markov chain with initial distribution $\mu$ and $\mathrm{Law}(\bar{X} | X) = \delta_{\nabla \varphi(X)}$ and $\mathrm{Law}(Y | \bar{X}) = q_{\bar{X}}$ almost surely, where $\nabla \varphi$ is the transport map given in Theorem \ref{main-result}. Then it holds
\begin{align*}
\E[|\E[Y|X]-X|^2] &= \E\left[\left|\E\left[\int y\,dq_{\bar{X}}(y)| X\right]-X\right|^2\right] = \E\left[\left|\int y\,dq_{\nabla \varphi(X)}(y)-X\right|^2\right]\\& = \E\left[\left|\nabla \varphi(X)-X\right|^2\right] = W_2^2(\bar{\mu},\mu) = \overline{\mathcal{T}}_2(\nu|\mu),
\end{align*}
where the last two equalities come respectively from \eqref{eq:gradphiopti} and Proposition \ref{prop:projection}. This shows the optimality of $(X,Y).$

Now let $d\pi(x,y)=d\mu(x)dp_x(y)$ be the law of the coupling $(X,Y)$ constructed above (with therefore $dp_x(y) = dq_{\nabla \varphi(x)}(y)$) let $d\pi'(x,y)=d\mu(x) dp'_x(y)$ be another weak optimal transport plan, then
\begin{align*}
\int \left | \int  y\,d\left(\frac{p_x+p'_x}{2}\right)(y) - x \right|^2 d\mu(x)\leq \int \frac{|\int y \,d p_x(y)-x|^2+|\int y \,d p'_x(y)-x|^2}{2}\,d\mu(x) = \overline{\mathcal{T}}_2(\nu|\mu).
\end{align*}
By optimality and strict convexity of $|\,\cdot\,|^2$ we deduce that $\int y\,dp_x=\int y\,dp'_x$ for $\mu$-almost every $x$. In other words,  $\E[Y'|X'=u]=\E[Y|X=u]=\nabla \varphi(u)$ for $\mu$ almost every $u\in \R^d.$  So if $(X',Y')$ is a weak optimal coupling, one has $\E[Y'|X'] = \nabla \varphi(X')$ almost surely. In particular $\E[Y'|X'] \sim \bar{\mu}$. The fact that $(\E[Y'|X'], Y')$ is a martingale is always true.
\end{proof}

\subsection{Proof of the technical Lemma \ref{lem:contraction}} For the sake of completeness, we recall the proof of this classical duality result for strongly convex functions. 

\proof[Proof of Lemma \ref{lem:contraction}]
Let us show that (a) implies (b). 
First let us show that $g^*$ is finite valued. If $x_o$ is some point where $g(x_o)<+\infty$, then it holds  $g^*(y) = \sup_{x\in \R^d} \{x\cdot y -g(x)\} \geq x_o\cdot y - g(x_o)$. This shows that $g^*$ does not take the value $-\infty$. On the other hand, the function $f(x) = g(x) - \frac{|x|^2}{2}$, $x \in \R^d$, is convex by assumption. So there exists some $a\in \R^d$ and $b\in \R$ such that $f(x) \geq a\cdot x+ b$ for all $x \in \R^d.$ It follows from this that $g(x) = f(x) + \frac{|x|^2}{2} \geq a\cdot x+ b + \frac{|x|^2}{2}$, $x \in \R^d$, which easily implies that $g^*(y)<+\infty$ for all $y \in \R^d.$ 
Since $g = f+ \frac{|\,\cdot\,|^2}{2}$ with $f$ convex, it follows from \textit{e.g.} \cite[Theorem E.2.3.1]{HUL01} that the convex conjugate of $g$ is given by 
\[
g^*(y) = \inf_{x\in \R^d} \left\{ f^*(x) + \frac{|y-x|^2}{2}\right\} = \frac{|y|^2}{2} - \sup_{x\in \R^d} \left\{ x\cdot y - \left(f^*(x) + \frac{|x|^2}{2}\right)\right\}
\]
The function defined by the supremum being clearly convex, it follows that $y\mapsto \frac{|y|^2}{2} - g^*(y)$ is convex on $\R^d$, which shows (b). Conversely, let us show that $(b) \Rightarrow (a)$. Let $k(y) = \frac{|y|^2}{2} - g^*(y)$, $y \in \R^d$, which is convex by assumption. By Fenchel-Legendre duality (see \textit{e.g} \cite[Corollary E 1.3.6]{HUL01}), it holds 
\[
g(x) = \sup_{y \in \R^d} \{x\cdot y - g^*(y)\} = \sup_{y\in \R^d}\left \{ x\cdot y - \frac{|y|^2}{2} +k(y)\right\},\qquad \forall x \in \R^d
\]
and so
\[
g(x) - \frac{|x|^2}{2} = \sup_{y \in \R^d}\left\{ k(y) - \frac{|x-y|^2}{2}\right\} = \sup_{u \in \R^d} \left\{k(x-u) - \frac{|u|^2}{2}\right\}\qquad \forall x \in \R^d.
\]
The function $x\mapsto g(x) - \frac{|x|^2}{2}$ is therefore convex as a supremum of convex functions.

 Now let us show that (a) implies (c). 
 We have already seen above that $g^*$ is finite valued over $\R^d.$ Since $g = f+ \frac{|\,\cdot\,|^2}{2}$ with $f$ convex, the function $g$ is also strictly convex. Therefore, according to \textit{e.g} \cite[Theorem E 4.1.1]{HUL01}, it follows that $g^*$ is of class $\mathcal{C}^1$ on $\R^d.$

 It remains to prove that $\nabla g^*$ is $1$-Lipschitz. Since the function $x\mapsto g(x) - \frac{|x|^2}{2}$ is convex, its subgradient is a monotone operator, which means that
\[
(b-a) \cdot (y-x) \geq |y-x|^2 ,\qquad \forall x,y \in \R^d,\qquad \forall b \in \partial g(y),\qquad  \forall a\in \partial g(x).
\]
Since $u \in \partial g(v)$ is equivalent to $v\in \partial g^*(u) = \{\nabla g^*(u)\}$ (see \textit{e.g.} \cite[Corollary E 1.4.4]{HUL01} and \cite[Corollary D 2.1.4]{HUL01}), the statement above is equivalent to
\[
(\nabla g^*(b) - \nabla g^*(a) ) \cdot (b-a) \geq |\nabla g^*(b) - \nabla g^*(a) |^2,\qquad \forall a,b \in \R^d,
\]
which immediately implies that $\nabla g^*$ is $1$-Lipschitz.
Finally, let us show that (c) implies (b). Since $\nabla g^*$ is $1$-Lipschitz, it holds
\[
(\nabla g^*(y) - \nabla g^*(x) ) \cdot (y-x) \leq |y-x|^2,\qquad \forall x,y \in \R^d,
\]
which easily implies that $x\mapsto \frac{|x|^2}{2} - g^*(x)$ is convex. This completes the proof.
\endproof

\bibliographystyle{abbrv}
\bibliography{basebib_mot}
\end{document}